\documentclass[10pt,a4paper]{article}
\usepackage[utf8]{inputenc}
\usepackage[T1]{fontenc}
\usepackage{geometry}
\usepackage{amsthm}
\usepackage[french,english]{babel}
\usepackage{amssymb}
\usepackage{lmodern}
\usepackage{amsmath,amsfonts}
\usepackage{mathrsfs} 
\usepackage{dsfont}
\usepackage{stmaryrd}
\DeclareMathOperator{\dive}{div} 
\usepackage{graphicx}
\usepackage{fullpage}
\usepackage[nottoc,notlot,notlof]{tocbibind}
\usepackage[colorlinks=true,urlcolor=black,linkcolor=black,citecolor=black]{hyperref}
\numberwithin{equation}{section}
\usepackage{mathtools,array}
\usepackage{verbatimbox}
\usepackage{color}
\usepackage{bm}
\usepackage{tikz-cd}
\usepackage{pst-node}
\usetikzlibrary{matrix}
\usepackage{mathtools}
\usepackage{verbatimbox}
\usepackage{diagbox}
\usepackage{array}
\newcolumntype{C}{>{$\displaystyle} c <{$}}
\usepackage{makecell}
\usepackage{float}
\usepackage{tabu}
\usepackage{cancel}
\usepackage{lscape}
\usepackage[babel=true]{csquotes}
\usepackage{bm}
\usepackage{xcolor}
\makeatletter
\def\env@dmatrix{\hskip -\arraycolsep
	\let\@ifnextchar\new@ifnextchar
	\def\arraystretch{2}%
	\array{*{\c@MaxMatrixCols}{>{\displaystyle}c}}%
}

\makeatother

\begin{document}

	\renewcommand{\thefootnote}{\fnsymbol{footnote}}
	
	\title{Higher Regularity of Weak Limits of Willmore Immersions II}
	\author{Alexis Michelat\footnote{Department of Mathematics, ETH Zentrum, CH-8093 Zürich, Switzerland.}\; and Tristan \selectlanguage{french}Rivière$^{*}$\selectlanguage{english}\setcounter{footnote}{0}}
	\date{\today}
	
	\maketitle
	
	\vspace{-0.5em}
	
	\begin{abstract}
		We obtain in arbitrary codimension a removability result on the order of singularity of Willmore surfaces realising the \emph{width} of Willmore min-max problems on spheres. As a consequence, out of the twelve families of non-planar minimal surfaces in $\mathbb{R}^3$ of total curvature greater than $-12\pi$, only three of them may occur as conformal images of bubbles in Willmore min-max problems. 
	\end{abstract}

	\tableofcontents
	\vspace{0cm}
	\begin{center}
	\end{center}
	
	\theoremstyle{plain}
	\newtheorem*{theorem*}{Theorem}
	\newtheorem{theorem}{Theorem}[section]
	\newenvironment{theorembis}[1]
	{\renewcommand{\thetheorem}{\ref{#1}$'$}%
		\addtocounter{theorem}{-1}%
		\begin{theorem}}
		{\end{theorem}}
	\renewcommand*{\thetheorem}{\Alph{theorem}}
	\newtheorem{lemme}[theorem]{Lemma}
	\newtheorem*{lemme*}{Lemma}
	\newtheorem{propdef}[theorem]{Definition-Proposition}
	\newtheorem*{propdef*}{Definition-Proposition}
	\newtheorem{prop}[theorem]{Proposition}
	\newtheorem{cor}[theorem]{Corollary}
	\theoremstyle{definition}
	\newtheorem*{definition}{Definition}
	\newtheorem{defi}[theorem]{Definition}
	\newtheorem{rem}[theorem]{Remark}
	\newtheorem*{rem*}{Remark}
	\newtheorem{rems}[theorem]{Remarks}
	\newtheorem{exemple}[theorem]{Example}
	\newtheorem{remsintro}{Remarks}
	\newtheorem{remintro}[remsintro]{Remark}
	\renewcommand\hat[1]{%
		\savestack{\tmpbox}{\stretchto{%
				\scaleto{%
					\scalerel*[\widthof{\ensuremath{#1}}]{\kern-.6pt\bigwedge\kern-.6pt}%
					{\rule[-\textheight/2]{1ex}{\textheight}}
				}{\textheight}%
			}{0.5ex}}%
		\stackon[1pt]{#1}{\tmpbox}
	}
	\parskip 1ex
	\newcommand{\totimes}{\ensuremath{\,\dot{\otimes}\,}}
	\newcommand{\vc}[3]{\overset{#2}{\underset{#3}{#1}}}
	\newcommand{\conv}[1]{\ensuremath{\underset{#1}{\longrightarrow}}}
	\newcommand{\A}{\ensuremath{\vec{A}}}
	\newcommand{\B}{\ensuremath{\vec{B}}}
	\newcommand{\C}{\ensuremath{\mathbb{C}}}
	\newcommand{\D}{\ensuremath{\nabla}}
	\newcommand{\E}{\ensuremath{\vec{E}}}
	\newcommand{\I}{\ensuremath{\mathbb{I}}}
	\newcommand{\Q}{\ensuremath{\vec{Q}}}
	\newcommand{\loc}{\ensuremath{\mathrm{loc}}}
	\newcommand{\z}{\ensuremath{\bar{z}}}
	\newcommand{\hh}{\ensuremath{\mathscr{H}}}
	\newcommand{\h}{\ensuremath{\vec{h}}}
	\newcommand{\vol}{\ensuremath{\mathrm{vol}}}
	\newcommand{\hs}[3]{\ensuremath{\left\Vert #1\right\Vert_{\mathrm{H}^{#2}(#3)}}}
	\newcommand{\R}{\ensuremath{\mathbb{R}}}
	\renewcommand{\P}{\ensuremath{\mathbb{P}}}
	\newcommand{\N}{\ensuremath{\mathbb{N}}}
	\newcommand{\Z}{\ensuremath{\mathbb{Z}}}
	\newcommand{\p}[1]{\ensuremath{\partial_{#1}}}
	\newcommand{\Res}{\ensuremath{\mathrm{Res}}}
	\newcommand{\lp}[2]{\ensuremath{\mathrm{L}^{#1}(#2)}}
	\renewcommand{\wp}[3]{\ensuremath{\left\Vert #1\right\Vert_{\mathrm{W}^{#2}(#3)}}}
	\newcommand{\wpn}[3]{\ensuremath{\Vert #1\Vert_{\mathrm{W}^{#2}(#3)}}}
	\newcommand{\np}[3]{\ensuremath{\left\Vert #1\right\Vert_{\mathrm{L}^{#2}(#3)}}}
	\newcommand{\hardy}[2]{\ensuremath{\left\Vert #1\right\Vert_{\mathscr{H}^{1}(#2)}}}
	\newcommand{\lnp}[3]{\ensuremath{\left| #1\right|_{\mathrm{L}^{#2}(#3)}}}
	\newcommand{\npn}[3]{\ensuremath{\Vert #1\Vert_{\mathrm{L}^{#2}(#3)}}}
	\newcommand{\nc}[3]{\ensuremath{\left\Vert #1\right\Vert_{C^{#2}(#3)}}}
	\renewcommand{\Re}{\ensuremath{\mathrm{Re}\,}}
	\renewcommand{\Im}{\ensuremath{\mathrm{Im}\,}}
	\newcommand{\dist}{\ensuremath{\mathrm{dist}}}
	\newcommand{\diam}{\ensuremath{\mathrm{diam}\,}}
	\newcommand{\leb}{\ensuremath{\mathscr{L}}}
	\newcommand{\supp}{\ensuremath{\mathrm{supp}\,}}
	\renewcommand{\phi}{\ensuremath{\vec{\Phi}}}
	\renewcommand{\H}{\ensuremath{\vec{H}}}
	\renewcommand{\L}{\ensuremath{\vec{L}}}
	\renewcommand{\lg}{\ensuremath{\mathscr{L}_g}}
	\renewcommand{\ker}{\ensuremath{\mathrm{Ker}}}
	\renewcommand{\epsilon}{\ensuremath{\varepsilon}}
	\renewcommand{\bar}{\ensuremath{\overline}}
	\newcommand{\s}[2]{\ensuremath{\langle #1,#2\rangle}}
	\newcommand{\pwedge}[2]{\ensuremath{\,#1\wedge#2\,}}
	\newcommand{\bs}[2]{\ensuremath{\left\langle #1,#2\right\rangle}}
	\newcommand{\scal}[2]{\ensuremath{\langle #1,#2\rangle}}
	\newcommand{\sg}[2]{\ensuremath{\left\langle #1,#2\right\rangle_{\mkern-3mu g}}}
	\newcommand{\n}{\ensuremath{\vec{n}}}
	\newcommand{\ens}[1]{\ensuremath{\left\{ #1\right\}}}
	\newcommand{\lie}[2]{\ensuremath{\left[#1,#2\right]}}
	\newcommand{\g}{\ensuremath{g}}
	\newcommand{\e}{\ensuremath{\vec{e}}}
	\newcommand{\f}{\ensuremath{\vec{f}}}
	\newcommand{\ig}{\ensuremath{|\vec{\mathbb{I}}_{\phi}|}}
	\newcommand{\ik}{\ensuremath{\left|\mathbb{I}_{\phi_k}\right|}}
	\newcommand{\w}{\ensuremath{\vec{w}}}
	\renewcommand{\v}{\ensuremath{\vec{v}}}
	\renewcommand{\tilde}{\ensuremath{\widetilde}}
	\newcommand{\vg}{\ensuremath{\mathrm{vol}_g}}
	\newcommand{\im}{\ensuremath{\mathrm{W}^{2,2}_{\iota}(\Sigma,N^n)}}
	\newcommand{\imm}{\ensuremath{\mathrm{W}^{2,2}_{\iota}(\Sigma,\R^3)}}
	\newcommand{\timm}[1]{\ensuremath{\mathrm{W}^{2,2}_{#1}(\Sigma,T\R^3)}}
	\newcommand{\tim}[1]{\ensuremath{\mathrm{W}^{2,2}_{#1}(\Sigma,TN^n)}}
	\renewcommand{\d}[1]{\ensuremath{\partial_{x_{#1}}}}
	\newcommand{\dg}{\ensuremath{\mathrm{div}_{g}}}
	\renewcommand{\Res}{\ensuremath{\mathrm{Res}}}
	\newcommand{\un}[2]{\ensuremath{\bigcup\limits_{#1}^{#2}}}
	\newcommand{\res}{\mathbin{\vrule height 1.6ex depth 0pt width
			0.13ex\vrule height 0.13ex depth 0pt width 1.3ex}}
	\newcommand{\ala}[5]{\ensuremath{e^{-6\lambda}\left(e^{2\lambda_{#1}}\alpha_{#2}^{#3}-\mu\alpha_{#2}^{#1}\right)\left\langle \nabla_{\vec{e}_{#4}}\vec{w},\vec{\mathbb{I}}_{#5}\right\rangle}}
	\setlength\boxtopsep{1pt}
	\setlength\boxbottomsep{1pt}
	\newcommand\norm[1]{%
		\setbox1\hbox{$#1$}%
		\setbox2\hbox{\addvbuffer{\usebox1}}%
		\stretchrel{\lvert}{\usebox2}\stretchrel*{\lvert}{\usebox2}%
	}
	\allowdisplaybreaks
	\newcommand*\mcup{\mathbin{\mathpalette\mcapinn\relax}}
	\newcommand*\mcapinn[2]{\vcenter{\hbox{$\mathsurround=0pt
				\ifx\displaystyle#1\textstyle\else#1\fi\bigcup$}}}
	\def\Xint#1{\mathchoice
		{\XXint\displaystyle\textstyle{#1}}%
		{\XXint\textstyle\scriptstyle{#1}}%
		{\XXint\scriptstyle\scriptscriptstyle{#1}}%
		{\XXint\scriptscriptstyle\scriptscriptstyle{#1}}%
		\!\int}
	\def\XXint#1#2#3{{\setbox0=\hbox{$#1{#2#3}{\int}$ }
			\vcenter{\hbox{$#2#3$ }}\kern-.58\wd0}}
	\def\ddashint{\Xint=}
	\newcommand{\dashint}[1]{\ensuremath{{\Xint-}_{\mkern-10mu #1}}}
	\newcommand\ccancel[1]{\renewcommand\CancelColor{\color{red}}\cancel{#1}}
	\newcommand\colorcancel[2]{\renewcommand\CancelColor{\color{#2}}\cancel{#1}}
	
	\section{Introduction}
	
	The main result of the paper is a removability result on the second residue (defined in \cite{beriviere}, see also the introduction of \cite{blow-up}) of branched Willmore spheres solving min-max problems (see \cite{eversion}).
	
    \begin{theorem}\label{B}
    	Let $n\geq 3$ and $\mathscr{A}$ be an admissible
    	family of $W^{2,4}$ immersions of the sphere $S^2$ into $\R^n$. Assume that
    	\begin{align*}
    	\beta_0=\inf_{A\in \mathscr{A}}\sup W(A)>0.
    	\end{align*}
    	Then there exists finitely many true branched compact Willmore spheres $\phi_1
    	,\cdots,\phi_p:S^2\rightarrow\R^n$, and true branched compact Willmore spheres $\vec{\Psi}_1,\cdots,\vec{\Psi}_q:S^2\rightarrow\R^n$ such that
    	\begin{align}\label{quantization2}
    	\beta_0=&\sum_{i=1}^{p}W(\phi_i)+\sum_{j=1}^{q}\left(W(\vec{\Psi}_j)-4\pi\theta_j\right)\in 4\pi\N,
    	\end{align}
    	where $\theta_0(\vec{\Psi}_j,p_j)\geq 1$ is the multiplicity of $\vec{\Psi}_j$ at some point $p_j\in \vec{\Psi}_j(S^2)\subset \R^n$. Then at every branch point $p$ of $\phi_i, \vec{\Psi}_j$ of multiplicity $\theta_0=\theta_0(p)\geq 2$, the second residue $r(p)$ satisfies the inequality $r(p)\leq \theta_0-2$.
    \end{theorem}

    As a corollary, we deduce as in \cite{blow-up} that most of possible minimal spheres of small absolute total curvature cannot occur. For the convenience of the reader, we recall the table below.
    
    {\tabulinesep=0.6mm
    	\begin{figure}[H]
    		\begin{center}
    			\begin{tabu}{|c|c|c|c|c|c|c|}
    				\hline
    				\begin{tabular}{@{}c@{}}Minimal\\ surface \end{tabular} & \begin{tabular}{@{}c@{}}Total \\ curvature\end{tabular}& \begin{tabular}{@{}c@{}}Non-zero \\ flux?\end{tabular} & \begin{tabular}{@{}c@{}}Number \\ of ends\end{tabular} & \begin{tabular}{@{}c@{}}Multiplicities\\ of ends\end{tabular} & \begin{tabular}{@{}c@{}}Second\\ residues \end{tabular} & \begin{tabular}{@{}c@{}}Possible min-max \\ Willmore bubble?\end{tabular}\\
    				\hline
    				\begin{tabular}{@{}c@{}}Catenoid\\  \end{tabular} & $-4\pi$ & Yes & $d=2$ & \begin{tabular}{@{}c@{}}  $m_1=1$\\ $m_2=1$  \end{tabular}  & \begin{tabular}{@{}c@{}} $r_1=0$\\ $r_2=0$ \end{tabular} & No \\
    				\hline
    				\begin{tabular}{@{}c@{}}Enneper \\ surface\end{tabular} & $-4\pi$ & No & $d=1$ & $m_1=3$ & $r_1=2$ & No\\
    				\hline
    				\begin{tabular}{@{}c@{}}Trinoid\\  \end{tabular} & $-8\pi$ & Yes & $d=3$ & \begin{tabular}{@{}c@{}} $m_j=1$\\ $1\leq j\leq 3$  \end{tabular} & \begin{tabular}{@{}c@{}} $r_j=0$\\ $1\leq j\leq 3$  \end{tabular} & No\\
    				\hline
    				\begin{tabular}{@{}c@{}}L\'{o}pez\\ surface I\end{tabular} & $-8\pi$ & No & $d=1$ & $m_1=5$ & $r_1=4$ & No\\
    				\hline
    				\begin{tabular}{@{}c@{}}L\'{o}pez\\ surface II\end{tabular} & $-8\pi$ & No & $d=1$ & $m_1=5$ & $r_1=3$ & Yes\\
    				\hline
    				\begin{tabular}{@{}c@{}}L\'{o}pez\\ surface III\end{tabular} & $-8\pi$ & Yes & $d=2$ & \begin{tabular}{@{}c@{}} $m_1=2$\\ $m_2=2$  \end{tabular} & \begin{tabular}{@{}c@{}} $r_1=1$\\ $r_2=1$  \end{tabular} & No\\
    				\hline
    				\begin{tabular}{@{}c@{}}L\'{o}pez\\ surface IV\end{tabular} & $-8\pi$ & Yes & $d=2$ & \begin{tabular}{@{}c@{}} $m_1=2$\\ $m_2=2$  \end{tabular} & \begin{tabular}{@{}c@{}} $r_1=1$\\ $r_2=1$  \end{tabular} & No\\
    				\hline
    				\begin{tabular}{@{}c@{}}L\'{o}pez\\ surface V\end{tabular} & $-8\pi$ & Yes & $d=2$ & \begin{tabular}{@{}c@{}} $m_1=2$\\ $m_2=2$  \end{tabular} & \begin{tabular}{@{}c@{}} $r_1=0$\\ $r_2=0$  \end{tabular} & No\\
    				\hline
    				\begin{tabular}{@{}c@{}}L\'{o}pez\\ surface VI\end{tabular} & $-8\pi$ & Yes & $d=2$ & \begin{tabular}{@{}c@{}} $m_1=1$\\ $m_2=3$  \end{tabular} & \begin{tabular}{@{}c@{}} $r_1=0$\\ $1\leq r_2\leq 2$  \end{tabular} & No \\
    				\hline
    				\begin{tabular}{@{}c@{}}L\'{o}pez\\ surface VII\end{tabular} & $-8\pi$ & Yes & $d=2$ & \begin{tabular}{@{}c@{}} $m_1=1$\\ $m_2=3$  \end{tabular} & \begin{tabular}{@{}c@{}} $r_1=0$\\ $r_2=2$  \end{tabular} & No\\
    				\hline
    				\begin{tabular}{@{}c@{}}L\'{o}pez\\ surface VIII\end{tabular} & $-8\pi$  & No & $d=2$ & \begin{tabular}{@{}c@{}} $m_1=1$\\ $m_2=3$  \end{tabular} & \begin{tabular}{@{}c@{}} $r_1=0$\\ $r_2=1$  \end{tabular} & Yes \\
    				\hline
    				\begin{tabular}{@{}c@{}}L\'{o}pez\\ surface IX\end{tabular} & $-8\pi$ & No & $d=2$ & \begin{tabular}{@{}c@{}} $m_1=1$\\ $m_2=3$  \end{tabular} & \begin{tabular}{@{}c@{}} $r_1=0$\\ $r_2=1$  \end{tabular} & Yes \\
    				\hline
    			\end{tabu}
    		\end{center}
    		\caption{Geometric properties of complete minimal surface with total curvature greater than $-12\pi$}\label{figure}
    	\end{figure}

    The main application of this Theorem is to restrict the possibilities of blow-up for Willmore surfaces realising the cost of the sphere eversion (\cite{kusner}).  

    \begin{theorem}\label{C} 
    Let $\iota_+:S^2\rightarrow\R^3$ be the standard embedding of the round sphere into Euclidean $3$-space, let $\iota_-:S^2\rightarrow \R^3$ be the antipodal embedding and let  $\mathrm{Imm}(S^2,\R^3)$ be the space of smooth \emph{immersions} from $S^2$ to $\R^3$. Furthermore, we denote by $\Omega$ the set of paths between the two immersions, defined by
    \begin{align*}
    \Omega=C^{0}([0,1],\mathrm{Imm}(S^2,\R^3))\cap\ens{\phi=\{\phi_t\}_{t\in [0,1]},\;\, \phi_0=\iota_+,\;\, \phi_1=\iota_-},
    \end{align*}
    and we define the \emph{cost of the min-max sphere eversion} by
    \begin{align*}
    \beta_0=\min_{\phi\in \Omega}\max_{t\in [0,1]}W(\phi_t)\geq 16\pi.
    \end{align*}	
    Let $\phi_1,\cdots,\phi_p,\vec{\Psi}_1,\cdots,\vec{\Psi}_q:S^2\rightarrow \R^3$ be branched Willmore spheres such that 
    \begin{align}\label{beta0}
    	\beta_0=\sum_{i=1}^{p}W(\phi_i)+\sum_{j=1}^q\left(W(\vec{\Psi}_j)-4\pi\theta_j\right)\geq 16\pi.
    \end{align}
    If $\beta_0=16\pi$, then we have either:
    \begin{enumerate}
    		\item[(1)] $p=1$, $q=0$ and $\phi_1$ is the inversion of a Bryant minimal sphere with four embedded planar ends.
    		\item[(2)] $1\leq p\leq 4$, $q=1$, $\phi_1,\cdots,\phi_p$ are round spheres and $\vec{\Psi}_1$ is the inversion of a Bryant minimal sphere with four embedded planar ends and $\theta_1=p$.
    	\end{enumerate}
    \end{theorem}
\begin{remsintro}
	\begin{enumerate}
		\item[(1)] 
	In the second case, the non-compact Willmore surface $\vec{\chi}_1:S^2\rightarrow \R^3$ arising in the bubble tree such that $W(\vec{\chi}_1)=W(\vec{\Psi}_1)-4\pi\theta_1$ is obtained by inverting $\vec{\Psi}_1$ at a point of multiplicity $\theta_1=p\in\ens{1,\cdots,4}$. It corresponds to a bubbling where a sphere is glued to each non-compact ends of $\vec{\chi}_1$ (there are exactly $p$ of them).
	
	\item[(2)] The inequality $\beta_0\geq 16\pi$ is a direct consequence of Li-Yau inequality (\cite{lieyau}) and of a celebrated result of Max and Banchoff \cite{max} (see also \cite{hughes}).
\end{enumerate}
\end{remsintro}
    \begin{proof}
    	By \cite{classification}, \cite{sagepaper}, ${\phi}_1,\cdots,{\phi}_p,\vec{\Psi}_1,\cdots, \vec{\Psi}_q:S^2\rightarrow \R^3$ are conformally minimal. 
    	
    	\textbf{First assume that not all $\phi_1,\cdots,\phi_p$ are round spheres.}
    	
    	Let $\n:S^2\setminus\ens{p_1,\cdots,p_d}\rightarrow S^2$ be the associated Gauss map of the dual minimal surface of (say) $\phi_1$, and $m_1,\cdots,m_d\geq 1$ be the respective multiplicities of the ends $p_1,\cdots,p_d$. 
    	
    	Thanks to the analysis of Theorem B of \cite{blow-up},
    	we first assume (using the Jorge-Meeks formula \cite{jorge})
    	\begin{align}\label{even}
    	0-1+\frac{1}{2}\sum_{j=1}^{d}(m_j+1)=\deg(\n)=\frac{1}{4\pi}\int_{S^2}-K_gd\vg\geq 3.
    	\end{align}
    	Therefore, we have
    	\begin{align*}
    	2\sum_{j=1}^{d}m_j\geq \sum_{j=1}^{d}(m_j+1)\geq 8.
    	\end{align*}
    	The conformal invariance of the Willmore energy coupled with the Li-Yau (\cite{lieyau}) inequality imply that  
    	\begin{align}\label{lieyau}
    	W(\phi_1)=4\pi\sum_{j=1}^{d}m_j\geq 16\pi.
    	\end{align}
    	Now, if 
    	\begin{align*}
    	\int_{S^2}K_gd\vg=-8\pi,
    	\end{align*}
    	then (by L\'{o}pez's classification \cite{lopez} and Figure \ref{figure}) $\phi_1$ is the inversion of a minimal sphere with one end of multiplicity $5$, which has Willmore  energy $W(\phi_1)=4\pi\times 5=20\pi$, or of a minimal sphere with two ends, one of multiplicity $3$ and the other planar (multiplicity $1$ with $0$ logarithmic growth), with Willmore energy $W(\phi_1)=4\pi\times (3+1)=16\pi$. Therefore, in all cases, we have
    	\begin{align*}
    	W(\phi_1)\geq 16\pi.
    	\end{align*}
    	In particular, $\beta_0=16\pi$ (or $\beta_0<32\pi$ with a minimal bubbling)
    	always implies that $p=1$ in \eqref{beta0}, and that the bubbling is minimal, \textit{i.e.} $W(\vec{\Psi}_j)=4\pi \theta_j$ for all $1\leq j\leq q$. 
    	
    	Now let $\vec{\chi}_1,\cdots,\vec{\chi}_q:S^2\rightarrow \R^3\cup\ens{\infty}$ be the dual minimal surfaces of $\vec{\Psi}_1,\cdots,\vec{\Psi}_q$, and let also $\tilde{\phi}_1:S^2\setminus \ens{p_1,\cdots,p_d}\rightarrow \R^3$ be the dual minimal surface of $\phi_1$. If $m_1,\cdots,m_d\geq 1$ are the multiplicities of the ends $p_1,\cdots,p_d$ of $\tilde{\phi}_1$, then by the Li-Yau inequality, we have (\cite{lieyau})
    	\begin{align*}
    		&W(\phi_1)=4\pi\sum_{j=1}^{d}m_j=16\pi,\qquad
    		\int_{S^2}K_{g_{\phi_1}}d\mathrm{vol}_{g_{\phi_1}}=4\pi+2\pi\sum_{j=1}^{d}(m_j-1).
    	\end{align*}
    	In particular, $d\leq 4$, and by the Jorge-Meeks formula (\cite{jorge}), $d$ must be even. 
    	
    	\textbf{Case 1.} $d=4$. Then by \cite{bryant}, $\phi_1$ is the inversion of Bryant's minimal surface with four embedded planar ends, and $q=0$.
    	
    	\textbf{Case 2.} $d=2$. Then by \cite{blow-up}, $\phi_1$ is the inversion of L\'{o}pez minimal surface with one planar end and one end of multiplicity $3$. Therefore, we deduce that
    	\begin{align*}
    		\int_{S^2}K_{g_{\phi_1}}d\mathrm{vol}_{g_{\phi_1}}=4\pi+2\pi(1-1)+2\pi(3-1)=8\pi,
    	\end{align*}
    	and the quantization of Gauss curvature shows that 
    	\begin{align*}
    		\sum_{j=1}^q\int_{S^2}K_{g_{\vec{\chi}_j}}d\mathrm{vol}_{g_{\vec{\chi}_j}}=-4\pi,
    	\end{align*}
    	but this is excluded by \cite{blow-up}. Indeed, this would correspond to the Enneper surface, which has one end of multiplicity $3$ and second residue $r=2=3-1$ forbidden by Theorem \ref{B}.
    	
    	\textbf{Now assume that $\phi_1,\cdots\phi_p$ are all round spheres.}
    	
    	Then $1\leq p\leq 4$, and 
    	\begin{align*}
    		\sum_{i=1}^p\int_{S^2} K_{g_{{\phi_i}}}d\mathrm{vol}_{g_{\phi_i}}=4\pi p.
    	\end{align*}
    	
    	\textbf{Case 1.} $p=4$. Then the bubbling is minimal, and by the quantization of the Gauss curvature (\cite{quanta})
    	\begin{align*}
    		\sum_{j=1}^q\int_{S^2}K_{g_{\vec{\Psi}_j}}d\mathrm{vol}_{g_{\vec{\Psi}_j}}=-12\pi.
    	\end{align*}
    	Therefore, we have $1\leq q\leq 3$ (counting only non-planar bubbles), and as by \cite{blow-up} (see Figure \ref{figure})
    	\begin{align*}
    		\sum_{j=1}^q\int_{S^2}K_{g_{\vec{\Psi}_j}}d\mathrm{vol}_{g_{\vec{\Psi}_j}}\geq -8\pi,
    	\end{align*}
    	we deduce that $q=1$ and 
    	\begin{align*}
    		\int_{S^2}K_{g_{\vec{\chi}_1}}d\mathrm{vol}_{g_{\vec{\chi}_1}}=-12\pi.
    	\end{align*}
    	Now, by the Jorge-Meeks formula, if $\vec{\chi}_1$ has $d$ ends of multiplicities $m_1,\cdots,m_d\geq 1$, we deduce that 
    	\begin{align*}
    		\int_{S^2}K_{g_{\vec{\chi}_1}}d\mathrm{vol}_{g_{\vec{\chi}_1}}=4\pi\left(0-1+\frac{1}{2}\sum_{j=1}^{d}\left(m_j+1\right)\right).
    	\end{align*}
    	In particular, we deduce that 
    	\begin{align*}
    		8=\sum_{j=1}^d(m_j+1)\geq 2d,
    	\end{align*}
    	or $d\leq 4$. Furthermore, as there are four bubbles of multiplicity $1$, we deduce that $\vec{\chi}_1$ must have at least four ends, so all of them are planar, with exactly says that $\vec{\chi}_1$ is a minimal sphere of Bryant with four planar ends. 
    	
    	\textbf{Case 2.} $p=3$. Then we have
    	\begin{align}\label{gaussquanta}
    		&\sum_{j=1}^qW(\vec{\chi}_j)=4\pi,\qquad
    		\sum_{j=1}^q\int_{S^2}K_{g_{\vec{\chi}_j}}d\mathrm{vol}_{g_{\vec{\chi}_j}}=-8\pi .
    	\end{align}
    	Here the $\vec{\chi}_j$ are not necessarily minimal and are the non-compact bubbles occurring in the min-max process. Here, without loss of generality, we can assume that $W(\vec{\chi}_1)=4\pi$. As $\vec{\chi}_1$ is not compact, we deduce that its dual minimal surface is non-planar. Furthermore, we deduce that $\vec{\chi}_2,\cdots,\vec{\chi}_q$ are minimal. In particular, for all $2\leq j\leq q$, either $\vec{\chi}_j$ is a plane, or
    	\begin{align}\label{gaussbonnet0}
    		\int_{S^2}K_{\vec{\chi}_j}d\mathrm{vol}_{\vec{\chi}_j}\leq -8\pi.
    	\end{align}
    	If $\vec{\xi}_1$ is the dual minimal surface of $\vec{\chi}_1$, we have by the conformal invariance of the Willmore energy
    	\begin{align}\label{negativegaussbonnet}
    		\int_{S^2}K_{g_{\vec{\chi}_1}}d\mathrm{vol}_{g_{\vec{\chi}_1}}=\int_{S^2}K_{g_{\vec{\xi}_1}}d\mathrm{vol}_{g_{\vec{\xi}_1}}+W(\vec{\chi}_1)\leq -8\pi+4\pi=-4\pi.
    	\end{align}
    	Therefore, comparing \eqref{gaussquanta}, \eqref{gaussbonnet0} and \eqref{negativegaussbonnet}, we deduce that $\vec{\chi}_2,\cdots,\vec{\chi}_q$ are all planes, and 
    	\begin{align*}
    		\int_{S^2}K_{g_{\vec{\xi}_1}}d\mathrm{vol}_{g_{\vec{\xi}_1}}=-12\pi.
    	\end{align*}
    	Furthermore, as $\phi_1$, ${\phi}_2$ and ${\phi}_3$ are round spheres, we deduce that all ends of $\vec{\xi}_1$ must be planar. Therefore, we deduce that $\vec{\chi}_1$ is the inversion of the compactification of Bryant surface with four planar ends at a point of multiplicity $3$ (if any).
    	
    	\textbf{Case 3.} $p=2$. Then we have
    	\begin{align}\label{gaussquanta2}
    	&\sum_{j=1}^qW(\vec{\chi}_j)=8\pi\nonumber \\
    	&\sum_{j=1}^q\int_{S^2}K_{g_{\vec{\chi}_j}}d\mathrm{vol}_{g_{\vec{\chi}_j}}=-4\pi .
    	\end{align}
    	We have either $W(\vec{\chi}_1)=8\pi$ and $\vec{\chi}_2,\cdots,\vec{\chi}_q$ are minimal, or $W(\vec{\chi}_1)=W(\vec{\chi}_2)=4\pi$. Notice that the first possibility shows as the dual minimal surface of $\vec{\chi}_1$ is not planar that (denote by $\vec{\xi}_1$ its dual minimal surface)
    	\begin{align*}
    		\int_{S^2}K_{g_{\vec{\chi}_1}}d\mathrm{vol}_{g_{\vec{\chi}_1}}=\int_{S^2}K_{g_{\vec{\xi}_1}}d\mathrm{vol}_{g_{\vec{\xi}_1}}+W(\vec{\chi}_1)\leq -8\pi+8\pi\leq 0
    	\end{align*}    	
    	while in the second case, 
    	\begin{align*}
    		&\int_{S^2}K_{g_{\vec{\chi}_1}}d\mathrm{vol}_{g_{\vec{\chi}_1}}\leq -4\pi,\qquad
    		\int_{S^2}K_{g_{\vec{\chi}_2}}d\mathrm{vol}_{g_{\vec{\chi}_2}}\leq -4\pi
    	\end{align*}
    	so the second case is excluded by \eqref{gaussbonnet0} and \eqref{gaussquanta2}. Therefore, we deduce as previously that $\vec{\xi}_1$ is the inversion of the compactification of Bryant surface with four planar ends at a point of multiplicity $2$ (if any).
    	
    	\textbf{Case 3.} $p=1$. Then we have
    	\begin{align}\label{gaussquanta3}
    	&\sum_{j=1}^qW(\vec{\chi}_j)=12\pi \nonumber\\
    	&\sum_{j=1}^q\int_{S^2}K_{g_{\vec{\chi}_j}}d\mathrm{vol}_{g_{\vec{\chi}_j}}=0 .
    	\end{align}
    	If $\vec{\chi}_j$ is not minimal and $W(\vec{\chi}_j)=4\pi m$, where $1\leq m\leq 3$, then we have with the previous notations
    	\begin{align*}
    		\int_{S^2}K_{g_{\vec{\chi}_j}}d\mathrm{vol}_{g_{\vec{\chi}_j}}=\int_{S^2}K_{g_{\vec{\xi}_j}}d\mathrm{vol}_{g_{\vec{\xi}_j}}+W(\vec{\chi}_j)\leq -8\pi+4\pi m=4\pi(m-2)\leq 4\pi, 
    	\end{align*}
    	so \eqref{gaussbonnet0} implies that the minimal $\vec{\chi}_j$ must be planes. 
    	
    	\textbf{Sub-case 1.} $W(\vec{\chi}_1)=12\pi$. Then we get as previously
    	\begin{align*}
    		\int_{S^2}K_{\vec{\xi}_1}d\mathrm{vol}_{g_{\vec{\xi}_1}}=-12\pi,
    	\end{align*}
    	and $\vec{\chi}_1$ is the inversion of the compactification of Bryant surface with four planar ends at a point of multiplicity $1$.
    	
    	\textbf{Sub-case 2.} $W(\vec{\chi}_1)=8\pi$, $W(\vec{\chi}_2)=4\pi$, then we get
    	\begin{align*}
    		\int_{S^2}K_{g_{\vec{\chi}_1}}d\mathrm{vol}_{g_{\vec{\chi}_1}}\leq 0,\qquad \int_{S^2}K_{g_{\vec{\chi}_1}}d\mathrm{vol}_{g_{\vec{\chi}_1}}\leq -4\pi,
    	\end{align*}
    	contradicting \eqref{gaussquanta3}.
    	
    	\textbf{Sub-case 3.} $W(\vec{\chi}_1)=W(\vec{\chi}_2)=W(\vec{\chi}_3)=4\pi$, so for all $1\leq j\leq 3$, we have
    	\begin{align*}
    		\int_{S^2}K_{g_{\vec{\chi}_1}}d\mathrm{vol}_{g_{\vec{\chi}_1}}\leq -4\pi,
    	\end{align*}
    	contradicting once more \eqref{gaussquanta3}.
    	
    	This analysis concludes the proof of the Theorem. 
    \end{proof}
    \begin{remsintro}
    	\begin{enumerate}
    	\item[(1)] Without Theorem \ref{B}, one could have had $p=q=1$ in \eqref{beta0}, where $\phi_1$ is the inversion of a minimal surface of L\'{o}pez with one planar end and one end of multiplicity $3$ (either the L\'{o}pez surface VIII or IX in Figure \ref{figure}), and $\vec{\Psi}_1$ is the inversion of Enneper surface, denoted by $\vec{\chi}_1$. Indeed, as 
    	\begin{align*}
    		&W(\phi_1)=4\pi(1+3)=16\pi,\qquad \int_{S^2}K_{\phi_1}d\mathrm{vol}_{g_{\phi_1}}=4\pi+2\pi(1-1)+2\pi(3-1)=8\pi,\\
    		&\int_{S^2}K_{\vec{\chi}_1}d\mathrm{vol}_{g_{\vec{\chi}_1}}=-4\pi
    	\end{align*}
    	and $\vec{\Psi}_1$ is minimal, we indeed have
    	\begin{align*}
    		&W(\phi_1)+\left(W(\vec{\Psi}_1)-12\pi\right)=W(\phi_1)+W(\vec{\chi}_1)=W(\phi_1)=16\pi\\
    		&\int_{S^2}K_{\phi_1}d\mathrm{vol}_{g_{\phi_1}}+\int_{S^2}K_{\vec{\chi}_1}d\mathrm{vol}_{g_{\vec{\chi}_1}}=4\pi,
    	\end{align*}	
    	so with respect to the quantization of energy this would have been a legitimate candidate. 
    	
    	\item[(2)] Recall by the non-existence of minimal surfaces with $5$ planar ends (\cite{bryant3}) that $\beta_0=20\pi$ implies that a non-trivial bubbling occurs. A possible \emph{minimal} bubbling is given by the L\'{o}pez minimal surface with one end of multiplicity $5$ and its inversion (see the L\'{o}pez surface II in Figure \ref{figure}).
    	
    	\item[(3)] We do not expect by the $\epsilon$-regularity depending only on the trace-free second fundamental form of Kuwert-Sch\"{a}tzle (\cite{kuwert3}) to have a limiting macroscopic surface $\phi_1$ to be a round sphere in case of bubbling. This would show that the only candidate with energy $\beta_0=16\pi$ is the compactification of Bryant's sphere with four planar ends. However, we cannot exclude for now this possibility.
    \end{enumerate}
    \end{remsintro}    
    
    \renewcommand{\thetheorem}{\thesection.\arabic{theorem}}
    
\section{The viscosity method for the Willmore energy}

We first introduce for all weak immersion $\phi:S^2\rightarrow \R^n$ of finite total curvature the associated metric $g=\phi^{\ast}g_{\R^n}$ on $S^2$. By the uniformisation theorem, there exists a function $\omega:S^2\rightarrow \R$ such that 
\begin{align*}
	g=e^{2\omega}g_0,
\end{align*}
where $g_0$ is a metric of constant Gauss curvature $4\pi$ and unit volume on $S^2$. Furthermore, in all fixed chart $\varphi:B(0,1)\rightarrow S^2$, we define $\mu:B(0,1)\rightarrow \R$ such that 
\begin{align*}
	\lambda=\alpha+\mu,
\end{align*}
where in the given chart 
\begin{align*}
	g=e^{2\lambda}|dz|^2.
\end{align*}
For technical reasons, we will have to make a peculiar choice of $\omega$ (see \cite{eversion}, Definition III.$2$).
\begin{defi}\label{aubin}
	Under the preceding notations, we say that a choice $(\omega,\varphi)$ of a map $\omega:S^2\rightarrow \R$ and of a diffeomorphism $\varphi:S^2\rightarrow S^2$ is an Aubin gauge if 
	\begin{align*}
		\varphi^{\ast}g_0=\frac{1}{4\pi}g_{S^2}\qquad \text{and}\qquad \int_{S^2} x_j e^{2\omega\circ \varphi(x)}d\mathrm{vol}_{g_{S^2}}(x)=0\qquad \text{for all}\;\, j=1,2,3,
	\end{align*}
	where $g_{S^2}$ is the standard metric on $S^2$.
\end{defi}
We also recall that the limiting maps arise from a sequence of critical point of the following regularisation of the Willmore energy (see \cite{eversion} for more details) : 
\begin{align*}
W_{\sigma}(\phi)&
=W(\phi)+\sigma^2\int_{S^2}\left(1+|\H|^2\right)^2d\vg\\
&+\frac{1}{\log\left(\frac{1}{\sigma}\right)}\left(\frac{1}{2}\int_{S^2}|d\omega|_g^2d\vg+4\pi\int_{S^2}\omega\, e^{-2\omega}d\vg-2\pi\log\int_{S^2}d\vg\right)
\end{align*}
where $\omega:S^2\rightarrow \R$ is as above.

We need a refinement of a standard estimate  (see \cite{helein}, $3.3.6$).

\begin{lemme}\label{l2weak}
	Let $\Omega$ be a open subset of $\R^2$ whose boundary is a finite union of $C^1$ Jordan curves. Let $f\in L^1(\Omega)$ and let $u$ be the solution of
	\begin{align}\label{system}
	\left\{\begin{alignedat}{2}
	\Delta u&=f\qquad &&\text{in}\;\, \Omega\\
	u&=0\qquad &&\text{on}\;\, \partial \Omega.
	\end{alignedat} \right.
	\end{align}
	Then $\D u\in L^{2,\infty}(\Omega)$, and 
	\begin{align*}
	\np{\D u}{2,\infty}{\Omega}\leq 3\sqrt{\frac{2}{\pi}}\np{f}{1}{\Omega}.
	\end{align*}
\end{lemme}
\begin{rem}
	We need an estimate independent of the domain for a sequence of annuli of conformal class diverging to $\infty$, but the argument applies to a general domain (although some regularity conditions seem to be necessary).
\end{rem}
\begin{proof}
	First assume that $f\in C^{0,\alpha}(\bar{\Omega})$ for some $0<\alpha<1$. Then by Schauder theory, $u\in C^{2,\alpha}(\bar{\Omega})$, and by Stokes theorem (\cite{hormandercomplex} $1.2.1$), we find as $u=0$ on $\partial \Omega$ that for all $z\in \Omega$
	\begin{align}\label{cauchy}
	\p{z}u(z)=\frac{1}{2\pi i}\int_{\Omega}\frac{\p{\z}\left(\p{z}u(\zeta)\right)}{\zeta-z}d\zeta\wedge d\bar{\zeta}.
	\end{align}
	As $\Delta u=4\,\p{z\z}^2u$ and $|d\zeta|^2=\dfrac{d\bar{\zeta}\wedge d\zeta}{2i}$,  the pointwise estimate \eqref{cauchy} implies that 
	\begin{align}\label{cauchy2}
	\p{z}u(z)=-\frac{1}{4\pi}\int_{\Omega}\frac{\Delta u(\zeta)}{\zeta-z}|d\zeta|^2=-\frac{1}{4\pi}\int_{\Omega}\frac{f(\zeta)}{\zeta-z}|d\zeta|^2.
	\end{align}
	Now, define $\bar{f}\in L^1(\R^2)$ by
	\begin{align*}
	\bar{f}(z)&=\left\{\begin{alignedat}{2}
	&f(z)\qquad &&\text{for all}\;\, z\in \Omega\\
	&0\qquad &&\text{for all}\;\, z\in \R^2\setminus \Omega.
	\end{alignedat}\right.
	\end{align*}
	and $U:\R^2\rightarrow \C$ by 
	\begin{align}\label{conv}
	U(z)=-\frac{1}{4\pi}\int_{\R^2}\frac{\bar{f}(\zeta)}{\zeta-z}|d\zeta|^2=-\frac{1}{4\pi}\left(\left(\zeta\mapsto\frac{1}{\zeta}\right)\ast \bar{f}\right)(z),
	\end{align}
	where $\ast$ indicates the convolution on $\R^2$. Now, recall that for all $1\leq p<\infty$ and $g\in L^p(\R^2,\C)$, we have
	\begin{align*}
	\np{\bar{f}\ast g}{p}{\R^2}\leq \np{\bar{f}}{1}{\R^2}\np{g}{p}{\R^2}.
	\end{align*}
	Interpolating between $L^1$ and $L^p$ for all $p>2$ shows by the Stein-Weiss interpolation theorem (\cite{helein} $3.3.3$) that for all $g\in L^{2,\infty}(\R^2,\C)$
	\begin{align*}
	\np{\bar{f}\ast g}{2,\infty}{\R^2}\leq \sqrt{2}\left(\frac{2\times 1}{2-1}+\frac{p\cdot 1}{p-2}\right)\np{\bar{f}}{1}{\R^2}\np{g}{2,\infty}{\R^2}=\sqrt{2}\left(2+\frac{p}{p-2}\right)\np{\bar{f}}{1}{\R^2}\np{g}{2,\infty}{\R^2}.
	\end{align*}
	Taking the infimum in $p>2$ (that is, $p\rightarrow \infty$) shows that for all $g\in L^{2,\infty}(\R^2)$, 
	\begin{align}\label{cauchy3}
	\np{\bar{f}\ast g}{2,\infty}{\R^2}\leq 3\sqrt{2}\np{\bar{f}}{1}{\R^2}\np{g}{2,\infty}{\R^2}.
	\end{align}
	Therefore, we deduce from \eqref{cauchy2} and \eqref{cauchy3} that 
	\begin{align*}
	\np{U}{2,\infty}{\R^2}\leq \frac{
		3\sqrt{2}}{4\pi}\np{\bar{f}}{1}{\R^2}\np{\frac{1}{|\,\cdot\,|}}{2,\infty}{\R^2}=\frac{3}{\sqrt{2\pi}}\np{f}{1}{\Omega}.
	\end{align*}
	Now, as $U=\p{z}u$ on $\Omega$ and $2|\p{z}u|=|\D u|$, we finally deduce that 
	\begin{align}\label{cauchy4}
	\np{\D u}{2,\infty}{\Omega}\leq 3\sqrt{\frac{2}{\pi}}\np{f}{1}{\Omega}.
	\end{align}
    In the general case $f\in L^1(\Omega)$, by density of $C^{\infty}_c(\Omega)$ in $L^1(\Omega)$, let $\ens{f_k}_{k\in \N}\subset C_{c}^{\infty}(\Omega)$ such that 
	\begin{align}\label{density}
	\np{f_k-f}{1}{\Omega}\conv{k\rightarrow \infty}0.
	\end{align}
	Then $u_k\in C^{\infty}(\bar{\Omega})$ (defined to be the solution of the system \eqref{system} with $f$ replaced by $f_k$ and the same boundary conditions) so for all $k\in \N$, $\D u_k\in L^{2,\infty}(\Omega)$ and
	\begin{align}\label{cauchy5}
	\np{\D u_k}{2,\infty}{\Omega}\leq 3\sqrt{\frac{2}{\pi}}\np{f_k}{1}{\Omega}.
	\end{align}
	As $\ens{\np{f_k}{1}{\Omega}}_{k\in \N}$ is bounded, up to a subsequence $u_k\underset{k\rightarrow\infty}{\rightharpoonup}  u_{\infty}$ in the weak topology of $W^{1,(2,\infty)}(\Omega)$. Therefore, \eqref{density} and \eqref{cauchy5} yield
	\begin{align*}
	\np{\D u_{\infty}}{2,\infty}{\Omega}\leq \liminf_{k\rightarrow \infty}\np{\D u_k}{2,\infty}{\Omega}\leq 3\sqrt{\frac{2}{\pi}}\np{f}{1}{\Omega}.
	\end{align*}
	Furthermore, as $f_k\conv{k\rightarrow \infty}f$ in $L^1(\Omega)$, we have $\Delta u_{\infty}=f$ in $\mathscr{D}'(\Omega)$, so we deduce that $u_{\infty}=u$ and this concludes the proof of the lemma.
\end{proof}

Finally, recall the following Lemma from \cite{quanta} (se also \cite{ge}).

\begin{lemme}\label{wentelpq}
	Let $\Omega$ be a Lipschitz bounded open subset of $\R^2$, $1<p<\infty$ and $1\leq q\leq \infty$, and $(a,b)\in W^{1,(p,q)}(B(0,1))\times W^{1,(2,\infty)}(B(0,1))$. Let $u:B(0,1)\rightarrow \R$ be the solution of 
	\begin{align*}
	\left\{\begin{alignedat}{2}
	\Delta u&=\D a\cdot \D^{\perp}b\qquad &&\text{in}\;\, \Omega\\
	u&=0\qquad &&\text{on}\;\, \partial \Omega.
	\end{alignedat} \right.
	\end{align*}
	Then there exists a constant $C_{p,q}(\Omega)>0$ such that 
	\begin{align*}
	\np{\D u}{p,q}{\Omega}\leq C_{p,q}(\Omega)\np{\D a}{p,q}{\Omega}\np{\D b}{2,\infty}{\Omega}.
	\end{align*}
\end{lemme}
\begin{rem}
	Notice that by scaling invariance, we have for all $R>0$ if $\Omega_R=B(0,R)$
	\begin{align*}
	\np{\D u}{2,1}{B(0,R)}\leq C_{2,1}(B(0,1))\np{\D a}{2,1}{B(0,R)}\np{\D b}{2,\infty}{B(0,R)}.
	\end{align*}
\end{rem}

\section{Improved energy quantization in the viscosity method}

\begin{theorem}\label{improvedquanta2}
	Under the hypothesis of Theorem \ref{B}, and  by \cite{eversion} let $\ens{\sigma_k}_{k\in \N}\subset (0,\infty)$ be such that $\sigma_k\conv{k\rightarrow \infty} 0$ and let $\{\phi_k\}_{k\in \N}:S^2\rightarrow \R^n$ be a sequence of critical points associated to $W_{\sigma_k}$ such that 
	\begin{align}\label{entropy}
	    \left\{\begin{alignedat}{1}     
	    &W_{\sigma_k}(\phi_k)=\beta(\sigma_k)\conv{k\rightarrow \infty}\beta_0\\
	    &W_{\sigma_k}(\phi_k)-W(\phi_k)=o\left(\frac{1}{\log\left(\frac{1}{\sigma_k}\right)\log\log\left(\frac{1}{\sigma_k}\right)}\right).
	    \end{alignedat}
	    \right.
	\end{align}
	Let $\ens{R_k}_{k\in \N}, \ens{r_k}_{k\in \N}\subset (0,\infty)$ be such that 
	\begin{align*}
		\lim_{k\rightarrow \infty}\frac{R_k}{r_k}=0,\qquad\limsup_{k\rightarrow\infty}R_k<\infty,
	\end{align*}
	and for all $0<\alpha<1$ and $k\in \N$, let $\Omega_k(\alpha)=B_{\alpha R_k}\setminus \bar{B}_{\alpha^{-1}r_k}(0)$ be a neck region, \textit{i.e.} such that 
	\begin{align*}
		\lim_{\alpha\rightarrow 0}\lim_{k\rightarrow \infty}\sup_{2\alpha^{-1}r_k<s<\alpha R_k/2}\int_{B_{2s}\setminus \bar{B}_{s/2}(0)}|\D\n_k|^2dx=0.
	\end{align*}
	Then we have
	\begin{align*}
		\lim_{\alpha\rightarrow 0}\limsup_{k\rightarrow \infty}\np{\D\n_k}{2,1}{\Omega_k(\alpha)}=0.
	\end{align*}
\end{theorem}
\begin{proof}
	As in \cite{eversion}, we give the proof in the special case $n=3$. 
	By Theorem $3.1$ of \cite{blow-up}
	, this is not restrictive. 
	\begin{align*}
		\Lambda=\sup_{k\in \N}\left(\np{\D\lambda_k}{2,\infty}{B(0,1)}+\int_{B(0,1)}|\D\n_k|^2dx\right)<\infty
	\end{align*}
	and
	\begin{align*}
		l(\sigma_k)=\frac{1}{\log\left(\frac{1}{\sigma_k}\right)},\qquad \tilde{l}(\sigma_k)=\frac{1}{\log\log\left(\frac{1}{\sigma_k}\right)}.
	\end{align*}
    Furthermore, the entropy condition \eqref{entropy} and the improved Onofri inequality show (see \cite{quanta} III.$2$)
    \begin{align}\label{ineq0}
    &\frac{1}{\log\left(\frac{1}{\sigma_k}\right)}\np{\omega_k}{\infty}{B(0,1)}=o\left(\frac{1}{\log\log\left(\frac{1}{\sigma_k}\right)}\right)\nonumber\\
    &\frac{1}{\log\left(\frac{1}{\sigma_k}\right)}\int_{S^2}|d\omega_k|_{g_k}^2d\mathrm{vol}_{g_k}=o\left(\frac{1}{\log\log\left(\frac{1}{\sigma_k}\right)}\right)\\
    &\frac{1}{\log\left(\frac{1}{\sigma_k}\right)}\left(\frac{1}{2}\int_{S^2}|d\omega_k|_{g_k}^2d\mathrm{vol}_{g_k}+4\pi\int_{S^2}\omega_ke^{-2\omega_k}d\mathrm{vol}_{g_k}-2\pi\log\int_{S^2}d\mathrm{vol}_{g_k}\right)=o\left(\frac{1}{\log\log\left(\frac{1}{\sigma_k}\right)}\right)\nonumber.
    \end{align}
    Thanks to \cite{eversion}, we already have
    \begin{align*}
    \lim_{\alpha\rightarrow 0}\limsup_{k\rightarrow \infty}\np{\D\n_k}{2}{\Omega_k(\alpha)}=0.
    \end{align*}
    Therefore, as in Lemma IV.$1$ in \cite{quanta} (and using the same argument as in Lemma $3.3$ of \cite{blow-up}
    ), there exists a controlled extension $\tilde{\n}_k:B(0,\alpha R_k)\rightarrow \mathscr{G}_{n-2}(\R^n)$ such that $\tilde{\n}_k=\n_k$ on $\Omega_k(\alpha)=B(0,\alpha R_k)\setminus\bar{B}(0,\alpha^{-1}r_k)$ and 
    \begin{align}\label{newcontrolledextension}
        &\np{\D\tilde{\n}_k}{2}{B(0,\alpha R_k)}\leq C_0(n)\np{\D\n_k}{2}{\Omega_k(\alpha)}\nonumber\\
        &\np{\D\tilde{\n}_k}{2,1}{B(0,\alpha R_k)}\leq C_0(n)\np{\D\n_k}{2,1}{\Omega_k(\alpha)}, 
    \end{align}
    in all equations involving $\n_k$ on $B(0,\alpha R_k)$, we replace $\n_k$ by $\tilde{\n}_k$ as one need only obtain estimates on $\Omega_k(\alpha)$, where $\tilde{\n}_k=\n_k$. Likewise, $\H_k$ can be replaced by a controlled extension using Lemma B.$4$ in \cite{quantamoduli} (see also the Appendix of \cite{blow-up}).
    
    Now, by \cite{eversion}, let 
    $\vec{L}_k:B(0,1)\rightarrow \R^3$ be such that
    \begin{align}\label{viscosityconservation1}
    d\vec{L}_k&=\ast\, d\left(\H_k+2\sigma_k^2(1+|\H_k|^2)\H_k\right)-2\left(1+2\sigma_k^2(1+|\H_k|^2)\right)H_k\,\ast d\n_k\nonumber\\
    &
    +\left(-\left(|\H_k|^2+\sigma_k^2(1+|\H_k|^2)^2\right)+\frac{1}{\log\left(\frac{1}{\sigma_k}\right)}\left(\frac{1}{2}|d\omega_k|_{g_k}^2-2\pi  \omega_ke^{-2\omega_k}+\frac{2\pi}{\mathrm{Area}(\phi_k(S^2))}\right)\right)\ast d\phi_k\nonumber\\
    &-\frac{1}{\log\left(\frac{1}{\sigma_k}\right)}\s{d\phi_k}{d\omega_k}_{g_k}\ast d\omega_k+\frac{1}{\log\left(\frac{1}{\sigma_k}\right)}\vec{\I}_k\res_{g_k}\left(\ast\, d\omega_k\right).
    \end{align}
    Then following \cite{eversion}, we have
    \begin{align*}
    	e^{\lambda_k(z)}|\vec{L}_k(z)|\leq \left(C_1(n)\left(1+\Lambda\right)e^{C_1(n)\Lambda}\np{\D\n}{2}{\Omega_k(\alpha)}+\tilde{l}(\sigma_k)\right)\frac{1}{|z|}\qquad \text{for all}\;\, z\in \Omega_k(\alpha/2),
    \end{align*}
    so that 
    \begin{align*}
    	\np{e^{\lambda_k}\vec{L}_k}{2,\infty}{\Omega_k(\alpha/2)}\leq 2\sqrt{\pi}\left(C_1(n)\left(1+\Lambda\right)e^{C_1(n)\Lambda}\np{\D\n_k}{2}{\Omega_k(\alpha)}+\tilde{l}(\sigma_k)\right).
    \end{align*}
    Now let $Y_k:B(0,\alpha R_k)\rightarrow \R$ (see \cite{eversion}, VI.$21$) be the solution of 
    \begin{align}\label{equation0}
    	\left\{\begin{alignedat}{2}
    	\Delta Y_k&=-4e^{2\lambda_k}\sigma_k^2\left(1-H_k^4\right)-2l(\sigma_k)K_{g_0}\omega_k e^{2\mu_k}+8\pi\,l(\sigma_k)e^{2\lambda_k}\mathrm{Area}(\phi(S^2))^{-1}\qquad &&\text{in}\;\, B(0,\alpha R_k)\\
    	Y_k&=0\qquad &&\text{on}\;\, \partial B(0,\alpha R_k).
    	\end{alignedat} \right.
    \end{align}
    Then we have (recall that $K_{g_0}=4\pi$ by the chosen normalisation in Definition \ref{aubin})
    \begin{align}\label{viscosity1}
    	\np{\Delta Y_k}{1}{B(0,\alpha R_k)}&\leq 4\sigma_k^2\int_{B(0,\alpha R_k)}\left(1+H_k^4\right)d\mathrm{vol}_{g_k}+8\pi\,l(\sigma_k)\np{\omega_k}{\infty}{B(0,\alpha_k)} \int_{B(0,\alpha R_k)}e^{2\mu_k}dx\nonumber\\
    	&+8\pi\,l(\sigma_k)\frac{\mathrm{Area}(\phi_k(B(0,\alpha R_k)))}{\mathrm{Area}(\phi_k(S^2))}=o(\tilde{l}(\sigma_k)).
    \end{align}
    Therefore, Lemma \ref{l2weak} implies by \eqref{viscosity1} that 
    \begin{align}\label{viscosity2}
    	\np{\D Y_k}{2,\infty}{B(0,\alpha R_k)}\leq 3\sqrt{\frac{2}{\pi}}\np{\Delta Y_k}{1}{B(0,\alpha R_k)}=o(\tilde{l}(\sigma_k))\leq \tilde{l}(\sigma_k)
    \end{align}
    for $k$ large enough. Now, let $\vec{v}_k:B(0,\alpha R_k)\rightarrow \R^3$ be the solution of
    \begin{align}\label{equation1}
    	\left\{\begin{alignedat}{2}
    	\Delta\vec{v}_k&=\D\tilde{\n}_k\cdot \D^{\perp}Y_k\qquad&&\text{in}\;\, B(0,\alpha R_k)\\
    	\vec{v}_k&=0\qquad &&\text{on}\;\, \partial B(0,\alpha R_k).
    	\end{alignedat}\right.
    \end{align}
    By scaling invariance and the inequality of Lemma \ref{wentelpq}, we deduce by \eqref{viscosity2} that for some universal constant $C_2>0$
    \begin{align}\label{viscosity3}
    	\np{\D\vec{v}_k}{2,1}{B(0,\alpha R_k)}&\leq C_2\np{\D\tilde{\n}_k}{2,1}{B(0,\alpha R_k)}\np{\D Y_k}{2,\infty}{B(0,\alpha R_k)}\nonumber\\
    	&\leq C_2C_0(n)\np{\D\n_k}{2,1}{\Omega_k(\alpha)}\np{\D Y_k}{2,\infty}{B(0,\alpha R_k)}\leq \tilde{l}(\sigma_k)\,\np{\D\n_k}{2,1}{\Omega_k(\alpha)}.
    \end{align}
    Furthermore, we have by Lemma \ref{wentelpq} and scaling invariance
    \begin{align}\label{viscositybis}
    	\np{\D\vec{v}_k}{2}{B(0,\alpha R_k)}&\leq C_3\np{\D\tilde{\n}_k}{2}{B(0,\alpha R_k)}\np{\D Y_k}{2,\infty}{B(0,\alpha R_k)}\leq C_3C_0(n)\np{\D\n_k}{2}{\Omega_k(\alpha)}o(\tilde{l}(\sigma_k))\nonumber\\
    	&\leq \tilde{l}(\sigma_k)\np{\D\n_k}{2}{\Omega_k(\alpha)}
    \end{align}
    Now, recall that  the Codazzi identity (\cite{eversion}, III.$58$) implies that 
    \begin{align}
    	\dive\left(e^{-2\lambda_k}\sum_{j=1}^2\I_{2,j}\p{x_j}\phi_k,-e^{-2\lambda_k}\sum_{j=1}^2\I_{1,j}\p{x_j}\phi_k\right)=0\qquad \text{in}\;\, B(0,\alpha R_k)
    \end{align}
    Therefore, by the Poincar\'{e} Lemma, there exists $\vec{D}_k:B(0,\alpha R_k)\conv{k\rightarrow \infty} \R^3$ such that 
    \begin{align*}
    	\D \vec{D}_k=\left(e^{-2\lambda_k}\sum_{j=1}^2\I_{1,j}\p{x_j}\phi_k,e^{-2\lambda_k}\sum_{j=1}^2\I_{2,j}\p{x_j}\phi_k\right).
    \end{align*}
    Notice that we have the trivial estimate
    \begin{align}\label{viscosity4}
    	\np{\D \vec{D}_k}{2}{B(0,\alpha_k)}\leq 2\np{\D\n_k}{2}{B(0,\alpha_k)}\leq 2\sqrt{\Lambda}.
    \end{align}
    Furthermore, 
    \begin{align}\label{viscosity5}
    	l(\sigma_k)\np{\D\vec{D}_k}{2,1}{\Omega_k(\alpha)}\leq 2\,l(\sigma_k)\,\np{\D\n_k}{2,1}{\Omega_k(\alpha)}.
    \end{align}
    Now, let $\vec{E}_k:B(0,\alpha R_k)\rightarrow \R^3$ be the solution of 
    \begin{align}\label{equation2}
    	\left\{
    	\begin{alignedat}{2}
    	\Delta \vec{E}_k&=2\,\D(l(\sigma_k)\omega_k)\cdot \D^{\perp}\vec{D}_k\qquad &&\text{in}\;\, B(0,\alpha R_k)\\
    	\vec{E}_k&=0\qquad &&\text{on}\;\, \partial B(0,\alpha R_k).
    	\end{alignedat}.
    	\right.
    \end{align}
    The improved Wente estimate, the scaling invariance and the estimates \eqref{entropy} and \eqref{viscosity4} imply that
    \begin{align}\label{viscosity6}
    	&\np{\D\vec{E}_k}{2,1}{B(0,\alpha R_k)}\leq 2C_0\,l(\sigma_k)\np{\D\omega_k}{2}{B(0,\alpha R_k)}\np{\D\vec{D}_k}{2}{B(0,\alpha R_k)}\leq 4C_0\sqrt{\Lambda}\,o(\sqrt{l(\sigma_k)})\leq \sqrt{l(\sigma_k)}\nonumber\\
    	&\np{\D\vec{E}_k}{2}{B(0,\alpha R_k)}\leq \frac{1}{2}\sqrt{\frac{3}{\pi}}\,l(\sigma_k)\np{\D\omega_k}{2}{B(0,\alpha R_k)}\np{\D\vec{D}_k}{2}{B(0,\alpha R_k)}\leq \sqrt{l(\sigma_k)}.
    \end{align}
    Now, let $\vec{F}_k:B(0,\alpha R_k)\rightarrow \R^3$ be such that 
    \begin{align*}
    	2\omega_k\,l(\sigma_k)\,\D^{\perp}\vec{D}_k=\D^{\perp}\vec{F}_k+\D\vec{E}_k.
    \end{align*}
    Combining \eqref{viscosity5}, \eqref{viscosity6}, and recalling that $l(\sigma_k)\np{\omega_k}{\infty}{B(0,\alpha R_k)}=o(\tilde{l}(\sigma_k))$ (by \eqref{entropy}), we deduce that 
    \begin{align}\label{viscosity7}
    	\np{\vec{F}_k}{2,1}{\Omega_k(\alpha)}&\leq 2\,l(\sigma_k)\np{\omega_k}{\infty}{\Omega_k(\alpha)}\np{\D\vec{D}_k}{2,1}{\Omega_k(\alpha)}+\np{\D\vec{E}_k}{2,1}{B(0,\alpha R_k)}\nonumber\\
    	&\leq \tilde{l}(\sigma_k)\np{\D\n_k}{2,1}{\Omega_k(\alpha)}+\sqrt{l(\sigma_k)}.
    \end{align}
    Finally, let $\vec{w}_k:B(0,\alpha R_k)\rightarrow \R^3$ be the solution of 
    \begin{align*}
    	\left\{\begin{alignedat}{2}
    	\Delta\vec{w}_k&=\D\tilde{\n}_k\cdot \D^{\perp}\left(\vec{v}_k-\vec{E}_k\right)\qquad&& \text{in}\;\, B(0,\alpha R_k)\\
    	\vec{w}_k&=0\qquad &&\text{on}\;\, \partial B(0,\alpha R_k).
    	\end{alignedat} \right.
    \end{align*}
    As previously, the improved Wente implies that 
    \begin{align}\label{viscosity8}
    	\np{\D\vec{w}_k}{2,1}{B(0,\alpha R_k)}&\leq C_0\np{\D\tilde{\n}_k}{2}{B(0,\alpha R_k)}\np{\D(\vec{v}_k-\vec{E}_k)}{2}{B(0,\alpha R_k)}\nonumber\\
    	&\leq C_0\np{\D\tilde{\n}_k}{2}{\Omega_k(\alpha)}\left(\np{\D\vec{v}_k}{2}{B(0,\alpha R_k)}+\np{\D\vec{E}_k}{2}{B(0,\alpha R_k)}\right)\nonumber\\
    	&\leq C_0C(n)\np{\D\n_k}{2}{\Omega_k(\alpha)}\left(\tilde{l}(\sigma_k)\np{\D\n_k}{2}{\Omega_k(\alpha)}+\sqrt{l(\sigma_k)}\right)\nonumber\\
    	&\leq C_0C(n)\sqrt{\Lambda}\left(\tilde{l}(\sigma_k)\sqrt{\Lambda}+\sqrt{l(\sigma_k)}\right)\leq \tilde{l}(\sigma_k)
    \end{align}
    for $k$ large enough. Finally, if $\vec{Z}_k:\Omega_k(\alpha)\rightarrow \R^3$ satisfies 
    \begin{align*}
    	\D^{\perp}\vec{Z}_k=\n_k\times \D^{\perp}\left(\vec{v}_k-\vec{E}_k\right)-\D\vec{w}_k,
    \end{align*}
    the estimates \eqref{viscosity3}, \eqref{viscosity6}, \eqref{viscosity8} show that (as $\tilde{\n}_k=\n_k$ on $\Omega_k(\alpha)$)
    \begin{align}\label{viscosity9}
    	\np{\D\vec{Z}_k}{2,1}{\Omega_k(\alpha)}&\leq \tilde{l}(\sigma_k)\np{\D\n_k}{2,1}{\Omega_k(\alpha)}+\sqrt{l(\sigma_k)}+\tilde{l}(\sigma_k).
    \end{align}
    Finally, following constants and using the controlled extension $\tilde{\n}_k$ of $\n_k$, we deduce as in \cite{eversion} (see (VI.$75$)) that 
    \begin{align}\label{viscosity10}
    	&\np{2\left(1+2\sigma_k^2\left(1+H_k^2\right)-l(\sigma_k)\omega_k\right)e^{\lambda_k}\H_k+\left(\D\vec{v}_k+\D^{\perp}\left(\vec{F}_k+\vec{Z}_k\right)\right)\times \D\phi_k\,e^{-\lambda_k}+l(\sigma_k)\D^{\perp}\vec{D}_k\cdot\D\phi_k\,e^{-\lambda_k}}{2,1}{\Omega_k(\alpha)}\nonumber\\
    	&\leq C_4(n)e^{C_4(n)\Lambda}\np{\D\n_k}{2}{\Omega_k(2\alpha)}.
    \end{align}
    Furthermore, as $l(\sigma_k)\np{\omega_k}{\infty}{\Omega_k(\alpha)}=o(\tilde{l}(\sigma_k))$, we have $2(1+2\sigma_k^2(1+H_k^2)-l(\sigma_k)\omega_k)\geq 1$ for $k$ large enough and by the estimates \eqref{viscosity3}, \eqref{viscosity5}, \eqref{viscosity7}, \eqref{viscosity9}, \eqref{viscosity10}, we deduce that 
    \begin{align}\label{viscosity11}
    &\np{e^{\lambda_k}\H_k}{2,1}{\Omega_k(\alpha)}
    \leq C_4(n)e^{C_4(n)\Lambda}\np{\D\n_k}{2}{\Omega_k(2\alpha)}
    +\tilde{l}(\sigma_k)\,\np{\D\n_k}{2,1}{\Omega_k(\alpha)}\nonumber\\
    &+\tilde{l}(\sigma_k)\np{\D\n_k}{2,1}{\Omega_k(\alpha)}+\sqrt{l(\sigma_k)}
    +\tilde{l}(\sigma_k)\np{\D\n_k}{2,1}{\Omega_k(\alpha)}+\sqrt{l(\sigma_k)}+\tilde{l}(\sigma_k)
    +2\,l(\sigma_k)\,\np{\D\n_k}{2,1}{\Omega_k(\alpha)}\nonumber\\
    &\leq C_4(n)e^{C_4(n)\Lambda}\np{\D\n_k}{2}{\Omega_k(2\alpha)}+5\,\tilde{l}(\sigma_k)\np{\D\n_k}{2,1}{\Omega_k(\alpha)}+3\,\tilde{l}(\sigma_k).
    \end{align}
    Thanks to the proof of Theorem $3.1$ and \eqref{viscosity11}, we have 
    \begin{align}\label{viscosity12}
    	\np{\D\n_k}{2,1}{\Omega_k(\alpha/2)}&\leq C_5(n)e^{C_5(n)\Lambda}\left(\np{\D\n_k}{2}{\Omega_k(\alpha)}+\np{e^{\lambda_k}\H_k}{2,1}{\Omega_k(\alpha)}\right)\nonumber\\
    	&\leq C_6(n)e^{C_6(n)\Lambda}\np{\D\n_k}{2}{\Omega_k(2\alpha)}+5\,\tilde{l}(\sigma_k)\np{\D\n_k}{2,1}{\Omega_k(\alpha)}+3\,\tilde{l}(\sigma_k).
    \end{align}
    Furthermore, thanks to the $\epsilon$-regularity (\cite{riviere1}), we obtain
    \begin{align}\label{viscosity13}
    	&\np{\D\n_k}{2,1}{B(0,\alpha R_k)\setminus \bar{B}(0,\alpha R_k/2)}\leq C_7(n)\np{\D\n_k}{2}{B(0,2\alpha R_k)\setminus \bar{B}(0,\alpha R_k/4)}\nonumber\\
    	&\np{\D\n_k}{2,1}{B(0,2\alpha^{-1}r_k)\setminus \bar{B}(0,\alpha^{-1}r_k)}\leq C_7(n)\np{\D\n_k}{2}{B(0,4\alpha^{-1}r_k )\setminus\bar{B}(0,\alpha^{-1} r_k/2)}.
    \end{align}
    Finally, by \eqref{viscosity12} and \eqref{viscosity13}, we have
    \begin{align*}
    	\np{\D\n_k}{2,1}{\Omega_k(\alpha)}\leq C_8(n)e^{C_8(n)\Lambda}\np{\D\n_k}{2}{\Omega_k(2\alpha)}+5\,\tilde{l}(\sigma_k)\np{\D\n_k}{2,1}{\Omega_k(\alpha)}+3\,\tilde{l}(\sigma_k),
    \end{align*}
    which directly implies as $\tilde{l}(\sigma_k)\conv{k\rightarrow \infty}0$ that for $k$ large enough
    \begin{align*}
    	\np{\D\n_k}{2,1}{\Omega_k(\alpha)}\leq 2C_8(n)e^{C_8(n)\Lambda}\np{\D\n_k}{2}{\Omega_k(2\alpha)}
    \end{align*}
    and the improved no-neck energy
    \begin{align*}
    	\lim_{\alpha\rightarrow  0}\limsup_{k\rightarrow \infty}\np{\D\n_k}{2,1}{\Omega_k(\alpha)}=0.
    \end{align*}
    This concludes the proof of the Theorem.
    \end{proof}
    
\section{Removability of the second residue}\label{sec6}

First, recall the following two Schwarz-type Lemmas from \cite{blow-up}. 

\begin{lemme}\label{schwarz}
	Let $0<4r<R<\infty$, let $\vec{u}:\Omega=B_R\setminus B_r(0)\rightarrow\C^m$ be a vector-valued holomorphic function and  let $\delta\geq 0$ be such that
	\begin{align*}
	&\np{\vec{u}}{\infty}{\partial B_r}\leq \delta.
	\end{align*}
	Then for all $1\leq j\leq m$, we have
	\begin{align*}
	|u_j(z)|\leq \frac{5}{R}\left(\np{u_j}{\infty}{\partial B_R}+\delta\right)|z|+2\delta.
	\end{align*}
\end{lemme}

\begin{lemme}\label{schwarz2}
	Let $0<r_1,\cdots,r_m<R<\infty$ be fixed radii, $a_1,\cdots,a_m\in B(0,R)$ be such that $\bar{B}(a_j,r_j)\subset B(0,R)$, $\bar{B}(a_j,r_j)\cap \bar{B}(a_k,r_k)=\varnothing$ for all $1\leq j\neq k\leq m$ and
	\begin{align}\label{hyp}
	4r_j<R-|a_j|\quad\text{for all}\;\, 1\leq j\leq m.
	\end{align} 
	Furthermore, define
	\begin{align*}
	\Omega=B(0,R)\setminus\bigcup_{j=1}^m\bar{B}(a_j,r_j)\qquad \Omega'=\bigcap_{j=1}^mB(a_j,R-|a_j|)\setminus \bar{B}(a_j,r_j).
	\end{align*}
	Let ${u}:\Omega\rightarrow \C$ be a holomorphic function and for all $1\leq j\leq m$ let $\delta_j\geq 0$ such that
	\begin{align*}
	\np{u}{\infty}{\partial B_{r_j}(a_j)}\leq \delta_j.
	\end{align*}
	Then we have for all $z\in \Omega'$
	\begin{align}\label{schwarzmultiple}
	|u(z)|&\leq \sum_{j=1}^m\frac{5}{R-|a_j|}\left(\frac{1}{m}\np{u}{\infty}{\partial B_R(0)}+2\delta_j\right)|z-a_j|+4\sum_{j=1}^{m}\delta_j\nonumber\\
	&+\sum_{j=1}^{m}\frac{5}{R-|a_j|}\left(\frac{1}{m}\np{u}{\infty}{\partial B_R(0)}+2\delta_j\right)\max_{1\leq j\neq k\leq m}\dist_{\mathscr{H}}(a_k,\partial B_{r_j}(a_j)),
	\end{align}
	where $\mathrm{dist}_{\mathscr{H}}$ is the Hausdorff distance.
\end{lemme}

\begin{theorem}\label{B2}
	Let $n\geq 3$ and $\mathscr{A}$ be an admissible
	family of $W^{2,4}$ immersions of the sphere $S^2$ into $\R^n$. Assume that
	\begin{align*}
	\beta_0=\inf_{A\in \mathscr{A}}\sup W(A)>0.
	\end{align*}
	Then there exists finitely many true branched compact Willmore spheres $\phi_1
	,\cdots,\phi_p:S^2\rightarrow\R^n$, and true branched compact Willmore spheres $\vec{\Psi}_1,\cdots,\vec{\Psi}_q:S^2\rightarrow\R^n$ such that
	\begin{align}\label{quantization2}
	\beta_0=&\sum_{i=1}^{p}W(\phi_i)+\sum_{j=1}^{q}\left(W(\vec{\Psi}_j)-4\pi\theta_j\right)\in 4\pi\N,
	\end{align}
	where $\theta_0(\vec{\Psi}_j,p_j)\geq 1$ is the multiplicity of $\vec{\Psi}_j$ at some point $p_j\in \vec{\Psi}_j(S^2)\subset \R^n$. Then at every branch point $p$ of $\phi_i, \vec{\Psi}_j$ of multiplicity $\theta_0=\theta_0(p)\geq 2$, the second residue $r(p)$ satisfies the inequality $r(p)\leq \theta_0-2$.
\end{theorem}
\begin{proof}
	Using the main result of \cite{eversion}, we see that there exists a sequence $\ens{\sigma_k}_{k\in \N}\subset (0,\infty)$ such that $\sigma_k\conv{k\rightarrow \infty}0$ and a sequence of smooth immersions $\{\phi_k\}_{k\in \N}\subset \mathrm{Imm}(S^2,\R^3)$ such that $\phi_k$ be a critical point of $W_{\sigma_k}$ such that 
	\begin{align*}
		W_{\sigma_k}(\phi_k)=\beta(\sigma_k),
	\end{align*}
	where 
	\begin{align*}
		\beta(\sigma_k)=\inf_{A\in \mathscr{A}}\sup W_{\sigma_k}(A)\conv{k\rightarrow \infty}\beta_0.
	\end{align*}
    Now, we can consider that $\phi_k:B(0,1)\rightarrow \R^3$ to be a critical point of $W_{\sigma_k}$ such that $\phi_k\conv{k\rightarrow \infty}\phi_{\infty}$ in $C^{l}_{\mathrm{loc}}(B(0,1)\setminus\ens{0})$ for all $l\in \N$. By \cite{beriviere}, there exists $\vec{C}_0\in \C^n$ such that 
    \begin{align*}
    	\H_{\phi_{\infty}}=\Re\left(\frac{\vec{C}_0}{z^{\theta_0-1}}\right)+O(|z|^{2-\theta_0}\log^2|z|).
    \end{align*}
    Then $r(0)\leq \theta_0-2$ if and only if $\vec{C}_0=0$, and this is what we will show in the rest of the proof.     
    Now, define $\Omega_k(\alpha)$ to be the neck-region
	\begin{align*}
		\Omega_k(\alpha)=B_{\alpha}(0)\setminus \bigcup_{j=1}^m\bar{B}_{\alpha^{-1}\rho_k^j}(x_k^j).
	\end{align*}
	Then notice that Theorem D of  \cite{blow-up} still applies. In particular, there exists $\theta_0^1,\cdots,\theta_0^m\in \Z\setminus\ens{0}$ and a universal constant $C=C(n,\Lambda)$ independent of $k$ and $0<\alpha<1$ such that for all $k\in \N$ large enough and for all $z\in \Omega_k(\alpha)$
	\begin{align}\label{holomorphicomparison}
	\frac{1}{C}\leq \frac{e^{\lambda_k(z)}}{\displaystyle\prod_{j=1}^{m}|z-x_k^j|^{\theta_0^j}}\leq C.
	\end{align}
	Now, let $\varphi_k(z):\Omega_k(\alpha)\rightarrow \C$ be the holomorphic function such that 
	\begin{align*}
		\varphi_k(z)=\prod_{j=1}^{m}\left(z-x_k^j\right)^{\theta_0^j}\qquad \text{for all}\;\, z\in \Omega_k(\alpha).
	\end{align*}
	To simplify notations, we state  the proof in codimension $1$ ($n=3$). 
	Furthermore, to simplify the proof, we will use Lemma \ref{schwarz} instead of \ref{schwarz2} and assume that we have only one bubble. We assume then in the following that 
	\begin{align*}
		\Omega_k(\alpha)=B_{\alpha}(0)\setminus \bar{B}_{\alpha^{-1}\rho_k}(0).
	\end{align*}
	
	First recall that the invariance by translation (\cite{eversion}, Lemma III.$8$) shows that there exists $\vec{L}_k:B(0,1)\rightarrow \R^3$ such that
	\begin{align}\label{viscosityconservation1}
		d\vec{L}_k&=\ast\, d\left(\H_k+2\sigma_k^2(1+|\H_k|^2)\H_k\right)-2\left(1+2\sigma_k^2(1+|\H_k|^2)\right)H_k\,\ast d\n_k\nonumber\\
		&
		+\left(-\left(|\H_k|^2+\sigma_k^2(1+|\H_k|^2)^2\right)+\frac{1}{\log\left(\frac{1}{\sigma_k}\right)}\left(\frac{1}{2}|d\omega|_{g_k}^2-2\pi  \omega_ke^{-2\omega_k}+\frac{2\pi}{\mathrm{Area}(\phi_k(S^2))}\right)\right)\ast d\phi_k\nonumber\\
		&-\frac{1}{\log\left(\frac{1}{\sigma_k}\right)}\s{d\phi_k}{d\omega_k}_{g_k}\ast d\omega_k+\frac{1}{\log\left(\frac{1}{\sigma_k}\right)}\vec{\I}_k\res_{g_k}\left(\ast\, d\omega_k\right).
	\end{align}
	Now, recall that $d=\partial+\bar{\partial}$, and $\ast\, \partial=-i\,\partial$, we deduce that
	\begin{align*}
		i\,\partial \vec{L}_k&=\partial\left((1+2\sigma^2_k\left(1+|\H_k|^2\right)\H_k\right)-2\left(1+2\sigma_k^2\left(1+|\H_k|^2\right)\right)H_k\,\partial\n_k\\
		&+\left(-\left(|\H_k|^2+\sigma_k^2(1+|\H_k|^2)^2\right)+\frac{1}{\log\left(\frac{1}{\sigma_k}\right)}\left(\frac{1}{2}|d\omega_k|_{g_k}^2-2\pi  \omega_ke^{-2\omega_k}+\frac{2\pi}{\mathrm{Area}(\phi_k(S^2))}\right)\right)\partial\phi_k\\
		&-\frac{1}{\log\left(\frac{1}{\sigma_k}\right)}\s{d\phi_k}{d\omega_k}_{g_k}\partial \omega_k+\frac{1}{\log\left(\frac{1}{\sigma_k}\right)}\vec{\I}_k\res_{g_k}\partial\omega_k.
	\end{align*}
	Finally, we can recast this equation as 
	\begin{align}\label{delbarre0}
		&\bar{\partial}\left(\left(1+2\sigma_k^2\left(1+|\H_k|^2\right)\right)\H_k+i\vec{L}_k\right)=-\left(1+2\sigma_k^2\left(1+|\H_k|^2\right)\right)H_k\,\bar{\partial}\n_k\nonumber\\
		&+\left(-\left(|\H_k|^2+\sigma_k^2\left(1+|\H_k|^2\right)^2\right)+\frac{1}{\log\left(\frac{1}{\sigma_k}\right)}\left(\frac{1}{2}|d\omega_k|_{g_k}^2-2\pi\omega_k e^{-2\omega_k}+\frac{2\pi}{\mathrm{Area}(\phi_k(S^2))}\right)\right)\bar{\partial}\phi_k\nonumber\\
		&-\frac{1}{\log\left(\frac{1}{\sigma_k}\right)}\s{d\phi_k}{d\omega_k}_{g_k}\bar{\partial}\omega_k+\frac{1}{\log\left(\frac{1}{\sigma_k}\right)}\vec{\I}_k\res_{g_k}\bar{\partial}\omega_k.
	\end{align}
	Now, let $\vec{\Psi}_k:\Omega_k(\alpha)\rightarrow \C$ be defined as 
	\begin{align}\label{psi} \vec{\Psi}_k(z)=\varphi_k(z)\left\{\left(1+2\sigma_k^2\left(1+|\H_k|^2\right)\right)\H_k+i\vec{L}_k\right\}.
	\end{align}
	Then by \eqref{delbarre0}, as $\varphi_k$ is holomorphic, $\vec{\Psi}_k$ solves
	\begin{align}\label{delbarre}
		&\bar{\partial}\vec{\Psi}_k=-\varphi_k\left(1+2\sigma_k^2\left(1+|\H_k|^2\right)\right)H_k\,\bar{\partial}\n_k\nonumber\\
		&+\varphi_k\left(-\left(|\H_k|^2+\sigma_k^2\left(1+|\H_k|^2\right)^2\right)+\frac{1}{\log\left(\frac{1}{\sigma_k}\right)}\left(\frac{1}{2}|d\omega_k|_{g_k}^2-2\pi\omega_k e^{-2\omega_k}+\frac{2\pi}{\mathrm{Area}(\phi_k(S^2))}\right)\right)\bar{\partial}\phi_k\nonumber\\
		&-\frac{1}{\log\left(\frac{1}{\sigma_k}\right)}\varphi_k\s{d\phi_k}{d\omega_k}_{g_k}\bar{\partial}\omega_k+\frac{1}{\log\left(\frac{1}{\sigma_k}\right)}\varphi_k\vec{\I}_k\res_{g_k}\bar{\partial}\omega_k.
	\end{align}
    Now, write $\vec{\Psi}_k=\vec{u}_k+\vec{v}_k+\vec{w}_k$, where
	\begin{align*}
		\left\{\begin{alignedat}{2}
		\bar{\partial}\vec{u}_k&=-\varphi_k H_k\bar{\partial}\n_k-\varphi_k|\H_k|^2\bar{\partial}\phi_k\qquad&&\text{in}\;\,\Omega_k(\alpha)\\
		\vec{u}_k&=0\qquad&&\text{on}\;\, \partial \Omega_k(\alpha)
		\end{alignedat}\right.
	\end{align*}
	and
	\begin{align}\label{eqw}
		\left\{\begin{alignedat}{2}
		\bar{\partial}\vec{w}_k&=-2\sigma_k^2\varphi_k\left(1+|\H_k|^2\right)H_k\bar{\partial}\n_k-\sigma_k^2\left(1+|\H_k|^2\right)^2\varphi_k\,\bar{\partial}\phi_k&&\text{in}\;\, \Omega_k(\alpha)\\
		&+\frac{1}{\log\left(\frac{1}{\sigma_k}\right)}\left(\frac{1}{2}|d\omega_k|_{g_k}^2-2\pi\omega_k e^{-2\omega_k}+\frac{2\pi}{\mathrm{Area}(\phi_k(S^2))}\right)\varphi_k\bar{\partial}\phi_k\\
		&-\frac{1}{\log\left(\frac{1}{\sigma_k}\right)}\s{d\phi_k}{d\omega_k}_{g_k}\varphi_k\bar{\partial}\omega_k+\frac{1}{\log\left(\frac{1}{\sigma_k}\right)}\varphi_k\left(\vec{\I}_k\res_{g_k}\bar{\partial}\omega_k\right)\qquad\\
		\vec{w}_k&=0\qquad&&\text{in}\;\,\partial \Omega_k(\alpha).
		\end{alignedat}\right.
	\end{align}
	Finally, $\vec{v}_k:\Omega_k(\alpha)\rightarrow \C$ is the holomorphic function
	\begin{align*}
		\left\{\begin{alignedat}{2}
		\bar{\partial}\vec{v}_k&=0\qquad&&\text{in}\;\, \Omega_k(\alpha)\\
		\vec{v}_k&=\vec{\Psi}_k\qquad&&\text{on}\;\, \partial\Omega_k(\alpha).
		\end{alignedat}\right.
	\end{align*}
	As $\bar{\partial}\vec{w}_k\in L^1(\Omega_k(\alpha))$, the Sobolev embedding only shows that  $\vec{w}_k\in L^{2,\infty}(\Omega_k(\alpha))$, so we have to obtain an estimate using this norm. 
	The duality $L^{2,1}/L^{2,\infty}$ shows that 
	\begin{align}\label{newnorm1}
		\np{\bar{\partial}\vec{u}_k}{1}{\Omega_k(\alpha)}&\leq \np{\varphi_k H_k}{2,\infty}{\Omega_k(\alpha)}\np{\bar{\partial}\n_k}{2,1}{\Omega_k(\alpha)}+\np{\varphi_kH_k}{2,\infty}{\Omega_k(\alpha)}\np{H_k\bar{\partial}\phi_k}{2,1}{\Omega_k(\alpha)}\nonumber\\
		&\leq 2\np{\D\n_k}{2,1}{\Omega_k(\alpha)}\np{\vec{\Psi}_k}{2,\infty}{\Omega_k(\alpha)}
	\end{align}
	This implies by Lemma \ref{l2weak} that
	\begin{align}\label{newnorm2}
		\np{\vec{u}_k}{2,\infty}{\Omega_k(\alpha)}\leq 6\sqrt{\frac{2}{\pi}}\np{\D\n_k}{2,1}{\Omega_k(\alpha)}\np{\vec{\Psi}_k}{2,\infty}{\Omega_k(\alpha)}
	\end{align}
	Finally, thanks to the maximum principle
	\begin{align}\label{maxprinciple}
		\np{\vec{v}_k}{\infty}{\Omega_k(\alpha)}\leq \np{\vec{\Psi}_k}{\infty}{\partial \Omega_k(\alpha)}.
	\end{align}
	Now, we will recall the estimates obtained in \cite{eversion} imply that 
	\begin{align}\label{ineq1}
		&\frac{1}{\log\left(\frac{1}{\sigma_k}\right)}\np{\omega_k}{\infty}{B(0,1)}=o\left(\tilde{l}(\sigma_k)\right)\nonumber\\
		&\frac{1}{\log\left(\frac{1}{\sigma_k}\right)}\int_{S^2}|d\omega_k|_{g_k}^2d\mathrm{vol}_{g_k}=o\left(\tilde{l}(\sigma_k)\right)\\
		&\frac{1}{\log\left(\frac{1}{\sigma_k}\right)}\left(\frac{1}{2}\int_{S^2}|d\omega_k|_{g_k}^2d\mathrm{vol}_{g_k}+4\pi\int_{S^2}\omega_ke^{-2\omega_k}d\mathrm{vol}_{g_k}-2\pi\log\int_{S^2}d\mathrm{vol}_{g_k}\right)=o\left(\tilde{l}(\sigma_k)\right)\nonumber.
	\end{align}
	Furthermore, as $\mathrm{Area}(\phi_k(S^2))=1$ and by \eqref{holomorphicomparison}, we deduce that 
	\begin{align}\label{ineq2}
		\frac{1}{\log\left(\frac{1}{\sigma_k}\right)}\np{\left(\frac{1}{2}|d\alpha|_{g_k}^2-2\pi\omega_k e^{-2\omega_k}+\frac{2\pi}{\mathrm{Area}(\phi_k(S^2))}\right)\varphi_k\bar{\partial}\phi_k}{1}{\Omega_k(\alpha)}=o\left(\tilde{l}(\sigma_k)\right),
	\end{align}
	while
	\begin{align}\label{ineq3}
		\frac{1}{\log\left(\frac{1}{\sigma_k}\right)}\np{\s{d\phi_k}{d\omega_k}\varphi_k\bar{\partial}\omega_k}{1}{\Omega_k(\alpha)}\leq \frac{C}{\log\left(\frac{1}{\sigma_k}\right)}\int_{\Omega_k(\alpha)}|d\omega_k|_{g_k}^2d\mathrm{vol}_{g_k}=o\left(\tilde{l}(\sigma_k)\right).
	\end{align}
	Finally, we have by the Cauchy-Schwarz inequality and \eqref{ineq1}
	\begin{align}\label{ineq4}
		\frac{1}{\log\left(\frac{1}{\sigma_k}\right)}\np{\varphi_k\left(\vec{\I}_k\res_{g_{k}}\bar{\partial}\omega_k\right)}{1}{\Omega_k(\alpha)}\leq \frac{C}{\log\left(\frac{1}{\sigma_k}\right)}\np{\D\n_k}{2}{\Omega_k(\alpha)}\np{\D\omega_k}{2}{\Omega_k(\alpha)}=o\left(\sqrt{l(\sigma_k)}\right).
	\end{align}
	Therefore, by \eqref{eqw},  \eqref{ineq2}, \eqref{ineq3} and \eqref{ineq4}
	\begin{align*}
		\np{\bar{\partial}\w_k}{1}{\Omega_k(\alpha)}=o\left(\tilde{l}(\sigma_k)\right)\conv{k\rightarrow \infty}0.
	\end{align*}
	Now, thanks to Lemma \ref{l2weak}, we deduce that 
	\begin{align}\label{2infty}
		\np{\w_k}{2,\infty}{\Omega_k(\alpha)}=o\left(\tilde{l}(\sigma_k)\right)\conv{k\rightarrow \infty}0.
	\end{align}
	Now, notice that 
	\begin{align*}
		\sigma_k e^{-\lambda_k}=\sigma_k \exp\left(o\left(\log\left(\frac{1}{\sigma_k}\right)\right)\right)=o_k(1)
	\end{align*}
	so that 
	\begin{align*}
		\sigma_k|\H_k|=\sigma_k e^{-\lambda_k}\,e^{\lambda_k}|\H_k|=o_k(1) e^{\lambda_k}|\H_k|.
	\end{align*}
	Therefore, by the conformal invariant of $e^{\lambda_k}|\H_k|$, the argument of 
    \textbf{Step 3} of the proof of Theorem $4.1$ (of \cite{blow-up}) still applies and shows that 
    \begin{align*}
    	|\vec{\Psi}_k|=O(\alpha^2)+o_k(1)\qquad \text{on}\;\, \partial B_{\alpha^{-1}\rho_k}(0).
    \end{align*}
    Therefore, 	Lemma \ref{schwarz} shows that there exists a universal constant $C_1=C_1(n,\Lambda)$ (independent of $k\in \N$ and $0<\alpha<1$) such that for all $z\in \Omega_k(\alpha)=B_{\alpha}\setminus \bar{B}_{\alpha^{-1}\rho_k}(0)$
    \begin{align}\label{schwarzineq}
        |\vec{v}_k(z)|\leq \frac{C_1}{\alpha}|z|+C_1(\alpha^2+o_k(1))
    \end{align}
    Now observe that 
    \begin{align*}
    	\int_{B_{\alpha}\setminus \bar{B}_{\alpha^{-1}\rho_k}}|z|^2|dz|^2=2\pi\int_{\alpha^{-1}\rho_k}^{\alpha}r^3dr\leq \frac{\pi}{2}\alpha^4.
    \end{align*}
    In particular, we deduce that 
    \begin{align}\label{newnorm3}
    	\np{\vec{v}_k}{2,\infty}{\Omega_k(\alpha)}\leq 2\np{\vec{v}_k}{2}{\Omega_k(\alpha)}\leq \sqrt{2\pi}C_1\alpha+\sqrt{\pi}C_1\alpha^2(\alpha^2+o_k(1))
    \end{align}
    Now, as $\vec{\Psi}_k=\vec{u}_k+\vec{v}_k+\vec{w}_k$, we have for fixed $0<\alpha<1$ and all $k\in\N$ large enough
    \begin{align*}
    	\np{\vec{\Psi}_k}{2,\infty}{\Omega_k(\alpha)}\leq \frac{6}{\sqrt{2\pi}}\np{\D\n_k}{2,1}{\Omega_k(\alpha)}\np{\vec{\Psi}_k}{2,\infty}{\Omega_k(\alpha)}+\sqrt{2\pi}C_1\alpha+\sqrt{\pi}C_1\alpha^2\left(\alpha^2+o_k(1)\right)+\frac{1}{\log\log\left(\frac{1}{\sigma_k}\right)}.
    \end{align*}
    Thanks to the improved no-neck energy of Theorem \ref{improvedquanta2}
    \begin{align}\label{newno-neck}
    	\lim\limits_{\alpha\rightarrow 0}\limsup_{k\rightarrow \infty}\np{\D\n_k}{2,1}{\Omega_k(\alpha)}=0,
    \end{align}
    we deduce that (for some $C_2=C_2(n,\Lambda)$)
    \begin{align*}
    	\np{\vec{\Psi}_k}{2,\infty}{\Omega_k(\alpha)}\leq C_2\alpha+C_2\alpha^2\left(\alpha^2+o_k(1)\right)+\frac{2}{\log\log\left(\frac{1}{\sigma_k}\right)}.
    \end{align*}
    Therefore, coming back to the estimate \eqref{newnorm2}, we deduce that for some universal constant $C_2=C_2(n,\Lambda)$
    \begin{align}\label{firstu}
    	\np{\vec{u}_k}{2,\infty}{\Omega_k(\alpha)}\leq \left(C_2\alpha+C_2\alpha^2\left(\alpha^2+o_k(1)\right)+\frac{2}{\log\log\left(\frac{1}{\sigma_k}\right)}\right)\np{\D\n_k}{2,1}{\Omega_k(\alpha)}.
    \end{align}
    Now, let $p=p(\alpha)>1$ be a fixed positive number independent of $k\in \N$, to be determined later. The estimate
    \begin{align}\label{integral}
    	\np{|z|}{2}{B_{\alpha^p}\setminus \bar{B}_{\alpha^{-1}\rho_k}(0)}^2=\int_{B_{\alpha^{p}}\setminus \bar{B}_{\alpha^{-1}\rho_k}(0)}|z|^{2}|dz|^2=2\pi\int_{\alpha^{-1}\rho_k}^{\alpha^p}r^3dr\leq \frac{\pi}{2}\alpha^{4p}
    \end{align}
    implies by \eqref{schwarzineq} that 
    \begin{align}\label{end1}
    	\np{\vec{v}_k}{2,\infty}{B_{\alpha^p}\setminus \bar{B}_{\alpha^{-1}\rho_k}(0)}&\leq \frac{2C_1}{\alpha}\times \sqrt{\frac{\pi}{2}}\alpha^{2p}+2\sqrt{\pi}C_1\alpha^p(\alpha^2+o_k(1))\nonumber\\
    	&\leq 2\sqrt{\pi}C_1\,\alpha^{2p-1}+2\sqrt{\pi}C_1\alpha^p\left(\alpha^2+o_k(1)\right).
    \end{align}
    Furthermore, the uniform $\epsilon$-regularity implies that for all $z\in B(0,1)\setminus \ens{0}$, we have
    \begin{align*}
    \lim\limits_{k\rightarrow \infty}\vec{\Psi}_k(z)=\vec{C}_0+O(|z|).
    \end{align*}
    Therefore, for all $\beta<\alpha^p$, we deduce thanks to the triangle inequality and \eqref{integral} that for all $k\in \N$ large enough
    \begin{align}\label{endter}
    	\np{\vec{\Psi}_k}{2,\infty}{B_{\alpha^p}\setminus \bar{B}_{\alpha^{-1}\rho_k}(0)}\geq \np{\vec{\Psi}_k}{2}{B_{\alpha^p}\setminus \bar{B}_{\beta}(0)}&\conv{k\rightarrow \infty}\np{\vec{C}_0+O(|z|)}{2,\infty}{B_{\alpha^p}\setminus \bar{B}_{\beta}(0)}\nonumber\\
    	&\geq \sqrt{\pi}\alpha^p\sqrt{1-\left(\frac{\beta}{\alpha^p}\right)^2}|\vec{C}_0|-C_3\alpha^{2p}
    \end{align}
    for some universal constant $C_3=C_3(n,\Lambda)$ independent of $0<\alpha<1$ and $p\geq 1$. Taking $\beta\rightarrow 0$ in \eqref{endter} yields
    \begin{align}\label{end2}
    \liminf\limits_{k\rightarrow \infty}\np{\vec{\Psi}_k}{2,\infty}{B_{\alpha^2}\setminus\bar{B}_{\alpha^{-1}\rho_k}(0)}\geq \sqrt{\pi}\alpha^{p}|\vec{C}_0|-C'\alpha^{2p}.
    \end{align}
    Furthermore, as $\vec{\Psi}_k=\vec{u}_k+\vec{v}_k+\vec{w}_k$, and as $\lim\limits_{k\rightarrow \infty  }\np{\vec{w}_k}{2,\infty}{\Omega_k(\alpha)}=0$ by \eqref{2infty}
    \begin{align}\label{end3}
    \limsup_{k\rightarrow\infty}\np{\vec{\Psi}_k}{2,\infty}{B_{\alpha^p}\setminus\bar{B}_{\alpha^{-1}\rho_k}(0)}\leq \limsup_{k\rightarrow\infty}\np{\vec{u}_k}{2,\infty}{B_{\alpha^p}\setminus\bar{B}_{\alpha^{-1}\rho_k}(0)}+2\sqrt{\pi}C_1\alpha^{2p-1}+2\sqrt{\pi}C_1\alpha^{p+2}.
    \end{align}
    Therefore, we deduce by \eqref{end2} and \eqref{end3}
    \begin{align}\label{end4}
    	\sqrt{\pi}\alpha^{p}|\vec{C}_0|-C_3\alpha^{2p}\leq \limsup_{k\rightarrow\infty}\np{\vec{u}_k}{2,\infty}{B_{\alpha^p}\setminus\bar{B}_{\alpha^{-1}\rho_k}(0)}+2\sqrt{\pi}C_1\alpha^{2p-1}+2\sqrt{\pi}C_1\alpha^{p+2}.
    \end{align}
    By \eqref{firstu}, we deduce that 
    \begin{align}\label{endbis}
    \sqrt{\pi}\alpha^{p}|\vec{C}_0|-C_3\alpha^{2p}\leq 2\,C_2\alpha\limsup_{k\rightarrow\infty}\np{\D\n_k}{2,1}{\Omega_k(\alpha)}+2\sqrt{\pi}C_1\alpha^{2p-1}+2\sqrt{\pi}C_1\alpha^{p+2}.
    \end{align}
    Therefore, we find by dividing both inequalities by $\sqrt{\pi}\alpha^p$ (using $p\geq 1$ for the second inequality) for some universal constant  $C_4=C_4(n,\Lambda)$
    \begin{align}\label{end11}
    	|\vec{C}_0|&\leq C_4\left(\frac{1}{\alpha^{p-1}}\limsup_{k\rightarrow\infty}\np{\D\n_k}{2,1}{\Omega_k(\alpha)}+\alpha^{p-1}+\alpha^p+\alpha^2\right)\nonumber\\
    	&\leq C_4\left(\frac{1}{\alpha^{p-1}}\limsup_{k\rightarrow\infty}\np{\D\n_k}{2,1}{\Omega_k(\alpha)}+\alpha^{p-1}+\alpha+\alpha^2\right).
    \end{align}
    We will now have to distinguish two cases
    
    \textbf{Case 1.}
    Now, if for all $0<\alpha<1$ small enough
    \begin{align*}
    	\limsup_{k\rightarrow\infty}\np{\D\n_k}{2,1}{\Omega_k(\alpha)}=0,
    \end{align*}
    then \eqref{end11} implies by taking $p=2$ that
    \begin{align*}
    	|\vec{C}_0|\leq C_4(2\alpha+\alpha^2)\conv{\alpha\rightarrow 0}0,
    \end{align*}
    which concludes the proof of the Theorem.
    
    \textbf{Case 2}. 
    Otherwise, assume that 
    \begin{align*}
    	f(\alpha)=\limsup_{k\rightarrow \infty}\np{\D\n_k}{2,1}{\Omega_k(\alpha)}>0\qquad\text{for all}\;\, 0<\alpha<1,
    \end{align*}
    and choose 
    \begin{align*}
    	p(\alpha)=1+\frac{1}{2}\frac{\log\left(\frac{1}{f(\alpha)}\right)}{\log\left(\frac{1}{\alpha}\right)}>1,
    \end{align*}
    so that 
    \begin{align*}
    	\alpha^{p-1}=\sqrt{f(\alpha)}.
    \end{align*}
    Then \eqref{end11} implies that 
    \begin{align*}
    	|\vec{C}_0|\leq C_4\left(\frac{1}{\sqrt{f(\alpha)}}\times f(\alpha)+\sqrt{f(\alpha)}+\alpha+\alpha^2\right)= C_4\left(2\sqrt{f(\alpha)}+\alpha+\alpha^2\right)\conv{\alpha\rightarrow 0}0
    \end{align*}
    and this concludes the proof of the Theorem by the improved no-neck energy \eqref{newno-neck}. 
    \end{proof}

\nocite{}
\bibliographystyle{plain}
\bibliography{biblio}

\end{document}